\documentclass[10pt]{amsart}
\usepackage{amssymb,mathrsfs,graphicx,extpfeil}
\usepackage{epsfig}
\usepackage{indentfirst, latexsym, amssymb, enumerate,amsmath,graphicx}
\usepackage{float}
\usepackage{colortbl}
\usepackage{epsfig,subfigure}
\usepackage{hyperref}
\usepackage{dsfont}
\usepackage[margin=1in]{geometry}

\usepackage{amsmath}
\title[On the isothermal Euler/Navier--Stokes system in a bounded domain]{Global-in-time dynamics of the two--phase fluid model in a bounded domain}

\author[Jung]{Jinwook Jung}
\address[Jinwook Jung]{\newline Research Institute of Basic Sciences \newline Seoul National University, Seoul  08826, Republic of Korea}
\email{warp100@snu.ac.kr}

\newtheorem{theorem}{Theorem}[section]
\newtheorem{lemma}{Lemma}[section]
\newtheorem{corollary}{Corollary}[section]
\newtheorem{proposition}{Proposition}[section]

\newtheorem{remark}{Remark}[section]

\newtheorem{definition}{Definition}[section]

\newcommand{\bbr}{\mathbb R}

\newcommand{\bbn} {\mathbb N}

\newcommand{\bq}{\begin{equation}}
\newcommand{\eq}{\end{equation}}
\newcommand{\N}{\mathbb{N}}
\newcommand{\ms}{\mathcal{S}}
\newcommand{\into}{\int_\Omega}
\newcommand{\mh}{\mathcal{H}}
\newcommand{\mt}{\mathcal{T}}

\newcommand{\e}{\varepsilon}

\newcommand{\R}{\mathbb R}
\newcommand{\lt}{\left}
\newcommand{\rt}{\right}
\newcommand{\om}{\Omega}
\newcommand{\pa}{\partial}
\makeatletter
\def\moverlay{\mathpalette\mov@rlay}
\def\mov@rlay#1#2{\leavevmode\vtop{%
   \baselineskip\z@skip \lineskiplimit-\maxdimen
   \ialign{\hfil$\m@th#1##$\hfil\cr#2\crcr}}}
\newcommand{\charfusion}[3][\mathord]{
    #1{\ifx#1\mathop\vphantom{#2}\fi
        \mathpalette\mov@rlay{#2\cr#3}
      }
    \ifx#1\mathop\expandafter\displaylimits\fi}
\makeatother

\newcommand{\mc}{\mathcal{C}}

\begin{document}
%%%%%%%%%%%%%%%%

\date{\today}

\subjclass{} \keywords{Global existence, Euler--Navier--Stokes system, kinematic boundary condition, two-phase fluid model, large-time behavior}

%\thanks{\textbf{Acknowledgment.} }
%
\begin{abstract}
In this work, we study the global existence of strong solutions and large-time behavior of a two-phase fluid model in a bounded domain. The model consists of the isothermal Euler equations and the isentropic compressible Navier--Stokes equations, coupled via the drag force. It was derived in \cite{CJpre} from a kinetic-fluid model describing the dynamics of particles subject to local alignment force and Brownian noises immersed in a compressible viscous fluid. For this system, we extend the local existence theory for strong solutions developed in \cite{CJpre} to obtain the global existence of strong solutions to the system. Moreover, we use the Lyapunov functional associated with the system to get large-time behavior estimates for global classical solutions.
\end{abstract}
\maketitle \centerline{\date}

\allowdisplaybreaks
\tableofcontents
\section{Introduction}
\setcounter{equation}{0}

In this paper, we are interested in a system of fluid equations in a bounded domain $\Omega \subset \R^3$ with smooth boundary $\pa\om$ and properties of its solutions, i.e. the global existence of strong solutions and large-time behavior. The system consists of the isothermal Euler equations and the isentropic compressible Navier-Stokes equations and they are coupled through the drag force. More specifically, let $\rho=\rho(x,t)$ and $n=n(x,t)$ be the densities of fluids at $(x,t) \in \Omega \times \R_+$ and $u=u(x,t)$ and $v=v(x,t)$ be the bulk velocities corresponding to $\rho$ and $n$, respectively. Then, the dynamics of a four-tuple $(\rho,u,n,v)$ is governed by the following two-phase fluid model: 
\begin{align}
\begin{aligned}\label{A-1}
&\partial_t \rho + \nabla_x \cdot (\rho u) = 0, \quad (x,t) \in \om \times \R_+,\\
&\partial_t (\rho u) + \nabla_x \cdot (\rho u \otimes u) + \nabla_x \rho = \rho(v-u),\\
&\partial_t n + \nabla_x \cdot (nv) = 0,\\
&\partial_t (nv) + \nabla_x \cdot (nv \otimes v) + \nabla_x p(n) - \Delta_x v  = -\rho(v-u), 
\end{aligned}
\end{align}
subject to initial and boundary conditions:
\begin{equation}\label{A-1_bdy}
\begin{cases}
(\rho(x,0), u(x,0), n(x,0), v(x,0)) = (\rho_0(x), u_0(x), n_0(x), v_0(x)), \quad \forall x \in \Omega, \\
u(x,t) \cdot r(x) \equiv 0,\quad \mbox{and} \quad v(x,t) \equiv 0, \quad \forall (x,t) \in \pa\Omega \times \R_+,
\end{cases}
\end{equation}
where $r=r(x)$ is the outward unit normal vector to $x \in \pa\om$. Here the pressure law is given by $p(n) = n^\gamma$ with $\gamma >1$. \\

System \eqref{A-1} can be derived from a kinetic-fluid model describing dynamics of particles with local alignment force and noises contained in a compressible viscous fluid.  Namely, let $f=f(x,\xi,t)$ be the number density of particles at the position $x \in \om $ with velocity $\xi \in \R^3$ at time $t > 0$, and $n = n(x,t)$  and $v = v(x,t)$ be the local density and velocity of the compressible viscous fluid, respectively. Then, we consider the kinetic-fluid model consisting of a Vlasov-Fokker-Planck equation and the isentropic compressible Navier-Stokes equations:

\begin{align}
\begin{aligned}\label{kin_flu}
&\partial_t f + \xi \cdot \nabla_x f + \nabla_\xi \cdot ((v-\xi) f) =  \nabla_\xi \cdot (\sigma\nabla_\xi f - \alpha(u - \xi)f), \quad (x,\xi,t) \in \om \times \R^3 \times \R_+,\\
&\partial_t n + \nabla_x \cdot (n v)=0, \quad (x,t) \in \om \times \R_+,\\
&\partial_t (n v) + \nabla_x \cdot (n v \otimes v) + \nabla_x p  -\mu\Delta_x v = -\int_{\bbr^3} (v-\xi) f \,d\xi, 
\end{aligned}
\end{align}
where $u$ is the local velocity of particles given by
\[u := \frac{\int_{\R^3} \xi f \,d\xi}{\int_{\R^3} f\,d\xi},\]
and the pressure law $p = p(n)$ is the same as in \eqref{A-1}. For boundary conditions, we set the homogeneous Dirichlet boundary condition for the viscous fluid velocity $v$ in \eqref{kin_flu}:
\[
v(x,t) = 0 \quad \mbox{for} \quad (x,t) \in \pa \om \times \R_+.
\]
To describe boundary conditions for the particle density $f$, we first set the outgoing/incoming boundaries written as
\[ 
\Sigma_{\pm} := \{ (x,\xi) \in \Sigma \, : \, \pm \,\xi \cdot r(x) >0\},
\]
respectively. We denote the traces of $f$ by $\gamma_{\pm} f(x,\xi,t) = f |_{\Sigma_{\pm}}$ and define
\[ 
L^p(\Sigma_{\pm}) := \left\{ g(x, \xi) \, : \, \left( \int_{\Sigma_{\pm}} |g(x,\xi)|^p |\xi \cdot r(x) | \,d\sigma (x) d\xi \right)^{1/p} < \infty \right\},   
\]
where $d\sigma(x)$ denotes the Euclidean metric on $\partial \Omega$. Now, on the particle density $f$, we impose the specular reflection boundary condition:
\bq\label{bdy_ref}
\gamma_- f(x,\xi,t) = \gamma_+ f (x, \mathcal{R}_x(\xi), t) \quad \mbox{for} \quad (x,\xi,t)\in\Sigma_- \times \R_+, 
\eq
where $\mathcal{R}_x(\xi) := \xi - 2(\xi\cdot r(x)) r(x)$ is the reflection operator. Note that this operator preserves the magnitude, i.e., $|\mathcal{R}_x (\xi)| = |\xi|$. \\

Here, we present a brief sketch  for derivation of our main system \eqref{A-1} from the kinetic-fluid model \eqref{kin_flu} (see \cite{CJpre} for details). First, we choose $\alpha = \sigma = \e^{-1}$ and let $(f^\e, n^\e, v^\e)$ be the corresponding solution to \eqref{kin_flu}.  Then we consider velocity moments of the kinetic equations
\[
\rho^\e (x,t) := \int_{\R^3} f^\e(x,\xi,t)\,d\xi, \quad (\rho^\e u^\e)(x,t) := \int_{\R^3} \xi f^\e(x, \xi, t)\,d\xi.
\]
Then the local densities $(\rho^\e,n^\e)$ and velocities $(u^\e,v^\e)$ would satisfy the following system:

$$\begin{aligned}
&\partial_t \rho^\e + \nabla_x \cdot (\rho^\e u^\e) = 0,\\
&\partial_t (\rho^\e u^\e) + \nabla_x \cdot (\rho^\e u^\e \otimes u^\e) + \nabla_x \cdot P^\e = \rho^\e(v^\e-u^\e),\\
&\partial_t n^\e + \nabla_x \cdot (n^\e v^\e) = 0,\\
&\partial_t (n^\e v^\e) + \nabla_x \cdot (n^\e v^\e \otimes v^\e) + \nabla_x p(n^\e) - \Delta_x v^\e  = -\rho^\e(v^\e-u^\e),
\end{aligned}$$
where the pressure $P^\e$ is given by
\[
P^\e(x,t) := \int_{\R^3 \times \R^3} (\xi-u^\e)\otimes (\xi-u^\e) f^\e\,d\xi.
\]
Then, we can show that $f^\e$ converges to the local Maxwellian as $\e \to 0$:
\[
f^\e(x,\xi,t) \to \frac{\rho(x,t)}{(2\pi)^{3/2}}e^{-\frac{|u(x,t)-\xi|^2}{2}},
\]
which also implies the convergence of the pressure $P^\e$:
\[
\int_{\R^3 \times \R^3} (\xi-u^\e) \otimes (\xi-u^\e) f^\e\,d\xi \to \rho \mathbb{I}_{3\times 3}, 
\]

Finally, we can derive the boundary condition for the limiting system as follows: First, we multiply the  specular reflection boundary condition \eqref{bdy_ref} by $(\xi \cdot r(x))$. Then, one integrates the resulting relation over the incoming boundary:
\bq\label{est_00}
\int_{\xi \cdot r(x) < 0} \gamma_- f^\e(x,\xi,t) (\xi \cdot r(x))\,d\xi = \int_{\xi \cdot r(x) < 0} \gamma_+ f^\e(x,\mathcal{R}_x(\xi),t) (\xi \cdot r(x))\,d\xi.
\eq
Next, we apply the change of variables $\xi_* = R_x(\xi)$ to \eqref{est_00} and combine the resulting relation with \eqref{est_00} to yield
\[
\int_{\R^3} \gamma f^\e(x,\xi,t) (\xi \cdot r(x))\,d\xi =0, \quad \mbox{i.e.,} \quad (\rho^\e u^\e)(x,t) \cdot r(x) = 0.
\]
Thus, taking the limit $\e \to 0$ leads to the kinematic boundary condition for the Euler equations in \eqref{A-1}:
\[
u \cdot r \equiv 0 \quad  \mbox{on} \quad \pa\om \times \R_+.
\]

\vspace{0.2cm}

Such kinetic-fluid models have been addressed in diverse perspectives due to possible applications in engineering, aerosols, medical sprays, etc \cite{BBJM05, BGLM15, Desv10, Oro81, Will58}. For example, the existence of weak solutions to Vlasov/incompressible Navier--Stokes systems is studied  for the spatial periodic domain \cite{BDGM} and the specular boundary condition \cite{Yu13}. Later, the large-time behavior estimates together with the global existence of solutions are considered in \cite{Hanpre, HMM20}.  For other types of kinetic equations with incompressible Navier--Stokes equations, Vlasov--Fokker--Planck/Navier--Stokes \cite{G-H-M-Z, G-J-V, G-J-V2}, Vlasov--Poisson/Navier--Stokes  \cite{AIK14, AKS10, CJpre2}, BGK/Navier--Stokes \cite{CLY21, CY20}, and Vlasov--Boltzmann/Navier--Stokes systems \cite{YY18}  are considered. For the coupling with compressible fluids, Vlasov/Euler \cite{BD06} and Vlasov--Fokker--Planck/Navier-Stokes equations \cite{CJpre, M-V, M-V2} are studied.\\

While kinetic-fluid models have been extensively studied so far, there have been few works on coupled fluid models. For instance, the isothermal Euler/incompressible and compressible Navier--Stokes equations \cite{C0, C},  the pressureless Euler/compressible Navier--Stokes equations \cite{CK16}, and  Euler--Poisson/incompressible Navier--Stokes equations \cite{CJpre2} are studied, under smallness assumptions on the smooth initial data. However, the previous results are mostly when the spatial domain is periodic or the whole space. When the spatial domain is bounded, there are numerous results for the existence of solutions to a sole compressible Euler or Navier--Stokes equations \cite{Ag, Eb79, F, FP00, FNP01, FJN12, L, PZ09, Sc86, Zh10}.  To the best of the author's knowledge, however, the global existence results of coupled fluid models in a bounded domain, especially the model consisting of the compressible Euler and Navier--Stokes equations have not been investigated before. \\

The results of this paper are two-fold. First, we extend the local existence theory in \cite{CJpre} to the global existence of system \eqref{A-1}. The difficulty of the problem mainly comes from the  kinematic boundary condition $u\cdot r \equiv 0$ imposed on the Euler equations. Due to the kinematic boundary condition, the usual energy estimates in $H^s(\Omega)$ do not work. In previous literature on compressible Euler equations in a bounded domain \cite{PZ09, Sc86, Zh10}, one needs $L^2$-estimates for time derivatives of $\rho$, $u$ and the curl of $u$ to cope with this problem. Thus, we also choose this approach to estimate the Euler equations in system \eqref{A-1}.  Moreover, this time derivative estimates requires different analyses on the drag force in the Navier-Stokes equations. In the previous literature on the global existence of Euler/Navier--Stokes systems \cite{C, CJpre2, CK16}, the authors first control the evolution of the viscous fluid velocity $v$ using the viscosity term in the Navier-Stokes equations. In turn, they used this to control $v$ in the drag force of the Euler parts.  However, In our system \eqref{A-1}, however, since we also need $L^2$-estimates for the time derivatives of $v$, the previous strategy needs some modification in its detail. Specifically, we carefully estimate $u$ in the drag force of the Navier--Stokes parts (see Lemma \ref{L2.5}) to get the desired dissipation estimates. \\

Second, we investigate the large-time behavior of global classical solutions to system \eqref{A-1}. Once the $L^\infty$-norm of $(\rho, u, n,v)$ are bounded uniformly in time, we can show that a Lyapunov functional consisting of the kinetic energy and fluctuations of densities around averages decays to 0 as time goes to infinity (see Theorem \ref{T1.2}). For this,  we consider an energy functional equivalent to the Lyapunov functional. Since the direct estimates for the energy functional do not give dissipation estimates with respect to densities $\rho$ and $n$, we use a Bogovskii-type inequality in a bounded domain to get the dissipation estimates for the perturbed energy functional, leading to the desired estimate for the Lyapunov functional. Furthermore, from the dissipation estimates for the Lyapunov functional,  we can obtain the convergence of strong solutions to \eqref{A-1} constructed in Theorem \ref{T1.1} toward the equilibrium in their solution space (see Corollary \ref{C2.1}).\\

The rest of this paper is organized as follows. In Section \ref{sec:2}, we present some definitions and notions used throughout this paper and summarize our main results for system \eqref{A-1}. In Section \ref{sec:3}, we show the global existence of strong solutions to \eqref{A-1}. In Section \ref{sec:4}, we provide the proof for the large-time behavior of global classical solutions to \eqref{A-1} and combine this result with the arguments in Section \ref{sec:3} to prove the large-time behavior of strong solutions in their solution space. Finally, Appendix \ref{app.A} is devoted to the proof of Lemma \ref{L2.5}.

\subsection{ Notation }  We summarize the notation used throughout this paper. For functions $f$, $\|f\|_{L^p}$ denote the usual $L^p(\om)$-norm. We also denote by $C$ a generic positive constant which may differ from line to line, and $C = C(\alpha,\beta,\dots)$ represents the positive constant depending on $\alpha,\beta,\dots$. For notational simplicity, we drop $x$-dependence of differential operators, i.e., $\nabla f := \nabla_x f$ and $\Delta f = \Delta_x f$. For any nonnegative integer $k$ and $p \in [1,\infty]$, $W^{k,p} := W^{k,p}(\Omega)$ stands for the $k$-th order $L^p$ Sobolev space. In particular, if $p=2$, we denote by $H^k := H^k(\Omega) = W^{k,2}(\Omega)$ for any $k \in \bbn$. Moreover, we let $\mathcal{C}^k(I;\mathcal{B})$ be the set of $k$-times continuously differentiable functions from an interval $I$ to a Banach space $\mathcal{B}$ and we introduce the following Banach spaces:
\[
\mathfrak{X}^s = \mathfrak{X}^s(T,\Omega) := \bigcap_{k=0}^{s}\mathcal{C}^k([0,T];H^{s-k}(\Omega)),\quad \mbox{and} \quad \|h(t)\|_{\mathfrak{X}^s}^2:=  \sum_{k=0}^{s} \left\| \frac{\partial^k f}{\partial t^k}(t)\right\|_{H^{s-k}}^2. 
\]
For functions $g=g(t)$ and $h=h(t)$ only depending on $t$, we write $g \approx h$ on $[0,T]$ if there exists a constant $C>0$ independent of $T$ such that
\[
\frac1C h(t) \le g(t) \le Ch(t), \quad \forall t \in [0,T].
\]
Finally, we denote $\nabla^\alpha$ ($\nabla_{t,x}^\alpha$, resp.) by a partial derivative with respect to $x$ (and $t$, resp.) with multi-index $\alpha \in (\N\cup\{0\})^3$ ($\alpha \in (\N\cup\{0\})^{3+1}$, resp.). 

\vspace{0.3cm}

\section{Main results}\label{sec:2}
\setcounter{equation}{0}

In this section, we introduce the definition of strong solutions to a system equivalent to \eqref{A-1} and state the result for the global-in-time existence of strong solutions and large-time behavior estimates. First of all, we reformulate the system \eqref{A-1} with new functions:  
\[
g := \log \lt(\frac{\rho}{\rho_c}\rt) \quad \mbox{with} \quad \rho_c = \into \rho\,dx > 0 \quad \mbox{and}  \quad h := n - \int_\Omega n \,dx =: n - n_c.
\]
Due to the conservation of mass, we have $\rho_c(t) = \rho_c(0)$ and $n_c(t) = n_c(0)$. Then formally, we can rewrite the system \eqref{A-1} as follows:
\begin{align}\label{A-2}
\begin{aligned}
&\partial_t g + u \cdot \nabla g + \nabla \cdot u = 0, \quad (x,t) \in \om \times \R_+,\\
&\partial_t u + u \cdot \nabla u + \nabla g = (v-u),\\
&\partial_t h + \nabla \cdot (hv) + n_c\nabla \cdot v = 0,\\
&\partial_t v + v \cdot \nabla  v + \frac{\nabla p(n_c+h)}{n_c+h} - \frac{1}{n_c+h}\Delta v  = -\frac{\rho_c e^{g}}{(n_c+h)}(v-u),
\end{aligned}
\end{align}
subject to initial data and boundary conditions:
\begin{align}\label{A-2_bdy}
\begin{aligned}
&(g(x,0), u(x,0), h(x,0), v(x,0)) = (g_0(x), u_0(x), h_0(x), v_0(x)), \quad x \in \Omega, \\
&u(x,t) \cdot r(x)  = 0, \quad v(x,t) = 0, \quad (x,t)\in \partial\Omega \times \R_+.
\end{aligned}
\end{align}
Without loss of generality, we set $\rho_c = n_c=1$. Now, we present the definition of strong solutions to the system \eqref{A-2}--\eqref{A-2_bdy} using the following solution space:
\[
\mh^s_T(\om) := \lt\{ (g,u,h,v) : (g,u,h,v) \in \mathfrak{X}^s(T,\om) \times [\mathfrak{X}^s(T,\om)]^3 \times \mathfrak{X}^s(T,\om) \times [\mathfrak{X}^s(T,\om)]^3 \rt\}.
\]

\begin{definition}\label{D2.2}
Let $T \in (0,\infty]$ and $s \ge 4$. We say a four-tuple $(g,u,h,v)$ is a strong solution to \eqref{A-2}--\eqref{A-2_bdy} on the time interval $[0,T]$ if it satisfies the following conditions:
\begin{itemize}
\item[(i)] $(g, u, h, v) \in \mh_T^s (\Omega)$. 
\item[(ii)] $(g,u,h,v)$ satisfies \eqref{A-2}--\eqref{A-2_bdy} in the distributional sense. 
\end{itemize}
\end{definition}

\begin{theorem}\label{T1.1}
Let $T \in (0,\infty)$ and $s\ge 4$. Suppose that the initial data $(g_0, u_0, h_0, v_0)$ satisfy the following regularity and smallness assumptions:
\begin{itemize}
\item[(i)] $(g_0, u_0, h_0, v_0) \in H^s(\Omega)\times H^s(\Omega)\times H^s(\Omega)\times H^s(\Omega)$.
\item[(ii)] The initial data satisfies the smallness condition:
\[
\|g_0\|_{H^s} + \|u_0\|_{H^s} + \|h_0\|_{H^s} + \|v_0\|_{H^s} < \e, 
\]
where $\e>0$ is a sufficiently small constant. 
\item[(iii)]
The initial data satisfies the compatibility conditions up to order s:
\[
\partial_t^k u(x,t) \cdot r(x)|_{t=0} =0, \quad \partial_t^k v(x,t)|_{t=0} =0, \quad x\in\partial\Omega, \quad k=0,1,\dots,s-1.
\]
\end{itemize}
Then, the initial-boundary value problem \eqref{A-2}--\eqref{A-2_bdy} has a unique solution 
$(g, u, h, v)\in \mh_\infty^s(\Omega)$ in the sense of Definition \ref{D2.2}.
\end{theorem}

\vspace{0.4cm}

\begin{remark}
\begin{enumerate}
\item
Here, the compatibility conditions $\partial_t^k u(x,t) \cdot r(x)|_{t=0} =0$ and $\partial_t^k v(x,t)|_{t=0} =0$ are actually given  recursively as follows:
\[
\begin{cases}
\displaystyle \pa_t^k u(x,t)\cdot r(x) |_{t=0} := -\pa_t^{k-1}(u \cdot \nabla u + \nabla g + u -v)(x,t)\cdot r(x) |_{t=0}=0,\\
\displaystyle \pa_t^k v(x,t)\cdot r(x)|_{t=0} := -\pa_t^{k-1} \lt( v \cdot \nabla v + \frac{\nabla p(1+h)}{1+h} -\frac{\Delta v}{1+h} + \frac{e^g}{1+h}(v-u)\rt)(x,t)\cdot r(x)|_{t=0}=0.
\end{cases}
\]

\item
Since $s\ge 4$, the solution $(g,u,h,v)$ obtained in Theorem \ref{T1.1} becomes a classical solution, i.e. $(g,u,h,v)\in\mc^2(\R_+ \times \om)$ and moreover, the smallness condition implies 
\[
\inf_{x \in \om}e^{g(x)}>0, \quad 1+\inf_{x\in\om} h(x) >0.
\]
Thus, the four-tuple $(\rho,u,n,v) \in \mc^2(\R_+ \times \om)$ solves the system \eqref{A-1}--\eqref{A-1_bdy}. Furthermore, we can deduce that $(\rho,u,n,v) \in \mh_\infty^s(\om)$. 
\end{enumerate}
\end{remark}

Next, we present the large-time behavior estimates for system \eqref{A-1}. To describe the result, we introduce a Lyapunov functional that measures the kinetic energy and fluctuations of densities around averages as follows:
\[
\mathscr{L}(t) := \frac12\into \rho|u|^2\,dx + \frac12\into n|v|^2\,dx + \into(\rho-\rho_c)^2\,dx + \into (n-n_c)^2\,dx.
\]
\begin{theorem}\label{T1.2}
Let $(\rho, u, n,v)$ be a global classical solution to system \eqref{A-1} satisfying the following conditions:

\begin{enumerate}
\item
$\rho \in [0,\bar\rho]$ and $n \in [0,\bar n]$ for some constants $\bar\rho, \bar n\in (0,\infty)$.
\vspace{0.1cm}

\item
 $ (u, v) \in [L^\infty(\bbr_+ \times \om)]^3 \times [L^\infty(\bbr_+ \times \om)]^3$.
\end{enumerate}
Then, we get
\[
\mathscr{L}(t) \le Ce^{-\eta_0  t}, \quad t \ge 0,
\]
where $C$ and $\eta_0$ are positive constants independent of $t$.
\end{theorem}
\vspace{0.4cm}

As a corollary of Theorem \ref{T1.1} and \ref{T1.2}, we actually find out that the (unique) global strong solution we established in Theorem \ref{T1.1} converges to the equilibrium $(\rho_\infty, u_\infty, n_\infty, v_\infty) = (\rho_c, 0 , n_c, 0)$ in their solution space $\mathfrak{X}^s$ as time goes to infinity.                                                                          

\begin{corollary}\label{C2.1}
Under the assumptions in Theorem \ref{T1.1}, we can obtain
\[
\lt( \|g(t)\|_{\mathfrak{X}^s}^2 + \|u(t)\|_{\mathfrak{X}^s}^2 + \|h(t)\|_{\mathfrak{X}^s}^2 + \|v(t)\|_{\mathfrak{X}^s}^2\rt) \le Ce^{-\eta t}, \quad t \ge 0,
\]
where $C$ and $\eta$ are independent of $t$.
\end{corollary}

%%%%%%%%%%%%%%%%%%%%%%%%%%%%%%%%%%%%%%%%%%%%%%%%%%%%%%%%%%%%
%
%
%
%. \section{Global well-posedness of the two-phase fluid system}\label{sec:2}
%
%
%%%%%%%%%%%%%%%%%%%%%%%%%%%%%%%%%%%%%%%%%%%%%%%%%%%%%%%%%%%%

\section{Global-in-time existence of the two-phase fluid system}\label{sec:3}
\setcounter{equation}{0}
In this section, we prove the global-in-time existence of the system \eqref{A-1}. Since the local well-posedness theory for the system \eqref{A-1} is established in \cite{CJpre}, we extend it to the global-in-time existence theory. First, the local well-posedness theory is as follows:

\begin{theorem}\label{thm_local}\cite{CJpre}
For $T>0$, there exists a small constant $\e_0>0$ such that if
\[
\|g_0\|_{H^s} + \|u_0\|_{H^s} + \|h_0\|_{H^s} + \|v_0\|_{H^s} <\e_0, 
\]
then a unique strong solution $(g,u,h,v) \in \mh^s_{T}(\om)$ of system \eqref{A-2} in the sense of Definition \ref{D2.2} corresponding to initial data $(g_0, u_0, h_0,v_0)$ exists up to time $t\le T$ and 
\[
\sup_{0 \leq t \leq T}\left(\|g(t)\|_{\mathfrak{X}^s} + \|u(t)\|_{\mathfrak{X}^s} + \|h(t)\|_{\mathfrak{X}^s}+\|h(t)\|_{\mathfrak{X}^s}\right) < \e_0^{1/2}.
\]
\end{theorem}

\subsection{A priori estimates for strong solutions} Here, we provide a priori estimates for strong solutions to system \eqref{A-2}. For this, we would use the following auxiliary functionals: for $p \in \N$,

\[
\begin{aligned}
&\mh(p; t) := \sup_{0 \le \tau \le t}\lt(  \|g(\tau)\|_{\mathfrak{X}^p}^2 +\|u(\tau)\|_{\mathfrak{X}^p}^2+\|h(\tau)\|_{\mathfrak{X}^p}^2+\|v(\tau)\|_{\mathfrak{X}^p}^2 \rt), \quad \mh_0(p) := \mh(p;0),\\
&\mt(p;t) :=  \lt(\sum_{\ell=0}^p\lt( \|\partial_t^\ell g(t)\|_{L^2}^2 + \|\partial_t^\ell u(t)\|_{L^2}^2\rt)  + \sum_{\ell=0}^{p-1} \|\partial_t^\ell \omega(t) \|_{H^{s-1-\ell}}^2 \rt),\\
&W(p;t):=  \lt(  \|g(t)\|_{\mathfrak{X}^p}^2 +\|u(t)\|_{\mathfrak{X}^p}^2 \rt),\\
&\mathcal{S}(p;t) :=  \lt(  \|h(t)\|_{\mathfrak{X}^p}^2 +\|v(t)\|_{\mathfrak{X}^p}^2 \rt),
\end{aligned}
\]
where $\omega := \nabla \times u$. Then, we have the following relation:

\begin{lemma}\label{L2.1}
For $T>0$, we can find a small positive constant $\e_p \ll 1$ such that if $\mh(s;T)\le \e_p^2 \ll 1$. Then for each $1 \le p \le s$ and $0<t \le T$, 
\[
W(p; t) \le C\lt(\mt(p;t) +  \|v(t)\|_{\mathfrak{X}^{p-1}}^2\rt),
\]
where $C=C(p,s,\Omega)$ is independent of $T$.
\end{lemma}
\begin{proof}
Since the proof is similar to that of \cite[Proposition 5.1]{CJpre} and \cite[Lemma 2.3]{Zh10}, we presnt a brief sketch for the proof. First, assume that
\[
\mh(s;T) \le \e_1^2 <1.
\]
the equation for $u$ in \eqref{A-2} gives
\bq\label{L3-1.1}\begin{aligned}
\|\nabla g\|_{L^2}^2 &= \|\partial_t u +u \cdot \nabla u + u-v \|_{L^2}^2\\
&\le C\lt(\|\partial_t u\|_{L^2}^2 +\|u\|_{L^\infty}^2 W(1;t) + \|u\|_{L^2}^2 + \|v\|_{L^2}^2\rt)\\
&\le C\lt(\|\partial_t u\|_{L^2}^2 +\e_1^2 W(1;t)  + \mathcal{T}(1;t) + \|v\|_{L^2}^2\rt)
\end{aligned}\eq
where we used Sobolev embedding $H^2(\om) \hookrightarrow \mc^0(\bar \om)$ and $C=C(s,\om) > 0$ is independent of $T$. Moreover, the equation for $g$ in \eqref{A-2} implies 
\[
\|\nabla \cdot u\|_{L^2}^2 \le C(\|\partial_t g\|_{L^2}^2 + \|u \cdot \nabla g\|_{L^2}^2) \le C\lt(\|\partial_t g\|_{L^2}^2 + \e_1^2 W(1;t)\rt),
\]
where $C=C(s,\Omega) > 0 $ is a positive constant. Now, we recall the following from \cite[Lemma 2.2]{Zh10} (see also \cite[Lemma 5]{B-B}): For ${\bf u} \in H^s(\Omega)$ with $ {\bf u} \cdot r \equiv 0$ on $\partial\Omega$, ${\bf u}$ satisfies
\[
\|{\bf u}\|_{H^s} \le C(\|\nabla \times {\bf u}\|_{H^{s-1}} + \|{\bf u}\|_{H^{s-1}} +\|\nabla \cdot {\bf u}\|_{H^{s-1}}).
\]
Hence, one gets
\bq\label{L3-1.2}\begin{aligned}
\|u\|_{H^1}^2 &\le C(\|\omega\|_{L^2}^2 + \|u\|_{L^2}^2 + \|\nabla \cdot u\|_{L^2}^2)\\
&\le C(\|\omega\|_{L^2}^2 + \|u\|_{L^2}^2 +\|\partial_t g\|_{L^2}^2 + \e_1^2 W(1;t))\\
&\le C(\mathcal{T}(1;t) + \e_1^2 W(1;t))
\end{aligned}\eq
where $C=C(s,\om) > 0$ is independent of $T$. Thus, we combine \eqref{L3-1.1} and \eqref{L3-1.2} to yield
\[\begin{aligned}
W(1;t) &=\lt(\|\pa_t g\|_{L^2}^2 + \|g\|_{H^1}^2 + \|\pa_t u\|_{L^2}^2 + \|u\|_{H^1}^2 \rt)\\
&\le \hat{C}\lt(\mathcal{T}(1;t) + \e_1^2 W(1;t) + \|v\|_{L^2}^2 \rt), 
\end{aligned}\]
where $\hat{C}=\hat{C}(s,\Omega) > 0$ is independent of $T$. So, if we choose $\e_1$ sufficiently small so that $\hat{C}\e_1^2 = \frac12$, the argument holds when $p=1$. From now on,  inductive arguments would yield the desired result.
\end{proof}
Now, we present $T$-independent estimates for $(g,u,h,v)$. First, we consider zeroth-order estimates.

\begin{proposition}\label{P2.1}
Assume that $\mh(s;T) \le \e^2 \ll 1$ for a sufficiently small $\e>0$ satisfying
\[
\sup_{0 \le t \le T} \|h(t)\|_{L^\infty} <\frac12, \quad \sup_{0 \le t \le T} \|g(t)\|_{L^\infty} <\log 2.
\]
Then, we have
\[
\mh(0;T) \le C\mh_0(0),
\]
where $C=C(\gamma)$ is independent of $T$.
\end{proposition}
\begin{proof}

From the original system \eqref{A-1}, straightforward computation gives
\begin{align*}
&\frac{d}{dt}\lt( \frac12 \int_\om \rho |u|^2\,dx + \int_\om \rho\log\rho\,dx \rt) = -\into \rho(u-v)\cdot u\,dx,\\
&\frac{d}{dt}\lt( \frac12 \int_\om n|v|^2\,dx+ \frac{1}{\gamma-1}\int_\om n^\gamma \,dx \rt)= -\into|\nabla v|^2\,dx +\into \rho(u-v)\cdot v\,dx.
\end{align*}
Thus, we can get
\begin{align*}
\frac{d}{dt}&\lt( \frac12 \int_\om \rho |u|^2\,dx + \frac12 \int_\om n|v|^2\,dx + \int_\om \rho\log\rho\,dx + \frac{1}{\gamma-1}\int_\om n^\gamma \,dx \rt) \\
&\qquad + \int_\om |\nabla v|^2\,dx + \int_\om \rho|u-v|^2\,dx = 0.
\end{align*}
Since we have
\[
\into \rho\log\rho\,dx = \into \lt(\rho\log\rho +1-\rho \rt)\,dx = \into ((g-1)e^g +1)\,dx, \]
we can use the second-order Taylor polynomial for the function $f(x) = (x-1)e^x$ at $x=0$ to attain
\[
\lt(1-\|g\|_{L^\infty}\rt)e^{-\|g\|_{L^\infty}} \|g\|_{L^2}^2 \le \into \rho\log\rho\,dx \le \lt(1+\|g\|_{L^\infty}\rt)e^{\|g\|_{L^\infty}} \|g\|_{L^2}^2.
\]
Moreover, since $|h|^2 = \lt(1-\frac{n_c}{n}\rt)^2 n^{2-\gamma} \cdot n^{\gamma}$, this gives  
\[
\frac14 \min\lt\{\lt(\frac12\rt)^{2-\gamma}, \lt(\frac32\rt)^{2-\gamma}  \rt\}\int_\om n^\gamma \,dx  \le \int_\om |h|^2\,dx \le  4\max\lt\{\lt(\frac12\rt)^{2-\gamma}, \lt(\frac32\rt)^{2-\gamma}  \rt\}\int_\om n^\gamma \,dx.
\]
Hence, we can obtain
\begin{align*}
&\lt(\|g(t)\|_{L^2}^2 + \|u(t)\|_{L^2}^2 + \|h(t)\|_{L^2}^2 + \|v(t)\|_{L^2}^2\rt)\\
&\quad \le C\lt(\frac12 \int_\om \rho |u|^2\,dx + \frac12 \int_\om n|v|^2\,dx + \int_\om \rho\log\rho\,dx + \frac{1}{\gamma-1}\int_\om n^\gamma \,dx \rt)\\
&\quad \le C\lt(\frac12 \int_\om \rho_0 |u_0|^2\,dx + \frac12 \int_\om n_0|v_0|^2\,dx + \int_\om \rho_0\log\rho_0\,dx + \frac{1}{\gamma-1}\int_\om n_0^\gamma \,dx \rt)\\
&\quad \le C\lt(\|g_0\|_{L^2}^2 + \|u_0\|_{L^2}^2 + \|h_0\|_{L^2}^2 + \|v_0\|_{L^2}^2\rt),
\end{align*}
where $C=C(\gamma)$ is independent of $T$. This implies the result.
\end{proof}

Next, we estimate time derivatives of $g$ and $u$ as follows.

\begin{lemma}\label{L2.2}
For $T>0$, suppose that $\mh(s;T)\le \e^2 \ll 1$ for a sufficiently small $\e>0$. Then for each $1 \le \ell \le s$ and $0<t \le T$, 
\[
\frac{d}{dt}\lt( \|\pa_t^\ell g\|_{L^2}^2 + \|\pa_t^\ell u\|_{L^2}^2 \rt) + \frac32\|\pa_t^\ell u\|_{L^2}^2 \le C\e  W(\ell;t) + 2\|\pa_t^\ell v\|_{L^2}^2,
\]
where $C=C(\ell, s, \Omega)$ is independent of $T$.
\end{lemma}
\begin{proof}
For $1\le \ell \le s$, we use Sobolev inequality to have

\begin{align*}
\frac12\frac{d}{dt}\|\pa_t^\ell g\|_{L^2}^2 &= \frac12 \int_\om (\nabla \cdot u) |\pa_t^\ell g|^2\,dx -\sum_{r=0}^{\ell-1} \binom{\ell}{r} \int_\om( \nabla (\partial_t^r g) \cdot \pa_t^{\ell-r} u) \pa_t^\ell g\,dx\\
&\quad -\int_\om \nabla \cdot (\pa_t^\ell u) \pa_t^\ell g\,dx\\
&\le C\e  W(\ell;t) -\int_\om \nabla \cdot (\pa_t^\ell u) \pa_t^\ell g\,dx,
\end{align*}
where $C=C(\ell,s, \Omega)$ is independent of $T$.\\

\noindent For $\pa_t^\ell u$, one uses Sobolev inequality and Young's inequality to get

\begin{align*}
\frac12\frac{d}{dt}\|\pa_t^\ell u\|_{L^2}^2 &= \frac12\int_\om (\nabla \cdot u) |\pa_t^\ell u|^2\,dx -\sum_{r=0}^{\ell-1} \binom{\ell}{r}\int_\om (\partial_t^{\ell-r} u \cdot \nabla (\pa_t^r u)) \cdot \pa_t^\ell u\,dx\\
&\quad - \int_\om \nabla(\pa_t^\ell g)\cdot \pa_t^\ell u\,dx -\int_\om (\pa_t^\ell u -\pa_t^\ell v)\cdot \pa_t^\ell u\,dx\\
&\le C\e  W(\ell,t) -\int_\om \nabla(\pa_t^\ell g)\cdot \pa_t^\ell u\,dx -\frac34\|\pa_t^\ell u\|_{L^2}^2 + \|\pa_t^\ell v\|_{L^2}^2.
\end{align*}
where $C=C(\ell,s, \Omega)$ is independent of $T$. Thus, we combine the previous estimates to get the desired result.
\end{proof}
Then, we aim to obtain the dissipation with respect to $g$.

\begin{lemma}\label{L2.3}
For $T>0$, suppose that $\mh(s;T)\le \e ^2 \ll 1$ for a sufficiently small $\e >0$. Then for each $1 \le \ell \le s$ and $0<t \le T$,
\[
-\frac{d}{dt}\int_\om (\partial_t^{\ell-1}g) (\partial_t^\ell g) \,dx +\|\partial_t^\ell g\|_{L^2}^2 \le C\e  W(\ell;t) + 2\|\partial_t^\ell u\|_{L^2}^2 + \frac34 \|\pa_t^{\ell-1} u\|_{L^2}^2 + \frac34 \|\pa_t^{\ell-1}v \|_{L^2}^2,
\]
where $C=C(\ell,s.\om)$ is independent of $T$.
\end{lemma}
\begin{proof}
Direct computation gives

\begin{align*}
-\frac{d}{dt}\int_\om (\partial_t^{\ell-1} g)( \partial_t^\ell g)\,dx &= -\|\pa_t^\ell g\|_{L^2}^2 -\int_\om (\partial_t^{\ell-1} g )(\partial_t^{\ell+1} g)\,dx\\
&=-\|\pa_t^\ell g\|_{L^2}^2 + \int_\om(\partial_t^{\ell-1} g )  \ \pa_t^\ell (\nabla g \cdot u + \nabla \cdot u)\,dx\\
&= -\|\pa_t^\ell g\|_{L^2}^2 +\sum_{r=0}^{\ell-1} \binom{l}{r} \int_\om (\partial_t^{\ell-1} g ) \nabla (\pa_t^r g) \cdot \pa_t^{\ell-r} u\,dx\\
&\quad + \int_\om (\pa_t^{\ell-1} g) \pa_t^\ell (\nabla \cdot u)\,dx + \int_\om (\partial_t^{\ell-1} g ) \nabla (\pa_t^\ell g)\cdot u\,dx\\
&=  -\|\pa_t^\ell g\|_{L^2}^2 +\sum_{r=0}^{\ell-1} \binom{l}{r} \int_\om \pa_t^{\ell-1} g \nabla (\pa_t^r g) \cdot \pa_t^{\ell-r} u\,dx\\
&\quad - \int_\om \nabla (\pa_t^{\ell-1} g) \cdot \pa_t^\ell  u\,dx -\int_\om \pa_t^\ell g \nabla \cdot (\pa_t^{\ell-1} g \ u) \,dx\\
&\le -\|\pa_t^\ell g\|_{L^2}^2 + C\e  W(\ell;t) + \|\pa_t^\ell u\|_{L^2}\|\nabla (\pa_t^{\ell-1} g)\|_{L^2},
\end{align*}
and from \eqref{A-2}, one gets

\[
\begin{aligned}
\|\nabla(\pa_t^{\ell-1} g)\|_{L^2} &= \|\pa_t^\ell u + \pa_t^{\ell-1} (u \cdot \nabla u) + \pa_t^{\ell-1}(u-v)\|_{L^2}\\
&\le \|\pa_t^\ell u\|_{L^2} + \|\pa_t^{\ell-1} u\|_{L^2} + \|\pa_t^{\ell-1} v\|_{L^2} + C\e  \sqrt{W(\ell;t)},
\end{aligned}
\]
where $C=C(\ell,s.\om)$ is independent of $T$. Thus, we use Young's inequality to obtain

\[
\begin{aligned}
-\frac{d}{dt}\int_\om( \partial_t^{\ell-1} g) (\partial_t^\ell g)\,dx &\le -\|\pa_t^\ell g\|_{L^2}^2 + C\e  W(\ell;t) \\
&\quad + \|\pa_t^\ell u\|_{L^2}\lt( \|\pa_t^\ell u\|_{L^2} + \|\pa_t^{\ell-1} u\|_{L^2} + \|\pa_t^{\ell-1} v\|_{L^2} + C\e  \sqrt{W(\ell;t)} \rt)\\
&\le -\|\pa_t^\ell g\|_{L^2}^2 + C\e  W(\ell;t)+ 2\|\pa_t^\ell u\|_{L^2}^2 + \frac34\|\pa_t^{\ell-1} u\|_{L^2}^2 + \frac34 \|\pa_t^{\ell-1} v\|_{L^2}^2,
\end{aligned}
\]
which implies the desired estimate.
\end{proof}
The next step is to estimate the curl of $u$.

\begin{lemma}\label{L2.4}
For $T>0$, suppose that $\mh(s;T)\le \e ^2 \ll 1$ for a sufficiently small $\e >0$. Then for each multi-index $\alpha$ with $|\alpha| = \ell-1$ and $1 \le \ell \le s$  and $0<t \le T$,
\[
\frac{d}{dt}\|\nabla_{t,x}^\alpha \omega\|_{L^2}^2 + \frac32 \|\nabla_{t,x}^\alpha\omega\|_{L^2}^2 \le C\e  W(\ell;t) + 2\|\nabla_{t,x}^\alpha \nu\|_{L^2}^2,
\]
where $\nu := \nabla \times v$ and $C=C(\ell,s,\om)$ is independent of $T$.
\end{lemma}
\begin{proof}
We apply $\nabla \times$ to \eqref{A-2}$_2$ to get

\[
\pa_t \omega + u \cdot \nabla \omega + \omega \cdot \nabla u = \nu-\omega,
\]
Thus, straightforward computation yields

\begin{align*}
\frac12\|\nabla_{t,x}^\alpha \omega\|_{L^2}^2 &= -\int_\om u \cdot \nabla (\nabla_{t,x}^\alpha \omega) \cdot \nabla_{t,x}^\alpha \omega\,dx -\int_\om \lt[\nabla_{t,x}^\alpha (u \cdot \nabla \omega) -u \cdot \nabla (\nabla_{t,x}^\alpha \omega)\rt] \cdot \nabla_{t,x}^\alpha \omega\,dx\\
&\quad - \int_\om \nabla_{t,x}^\alpha (\omega\cdot \nabla u)\cdot \nabla_{t,x}^\alpha \omega \,dx- \int_\om \nabla_{t,x}^\alpha (\omega-\nu)\cdot \nabla_{t,x}^\alpha \omega\,dx\\
&\le C\e  W(\ell;t) -\frac34\|\nabla_{t,x}^\alpha \omega\|_{L^2}^2 + \|\nabla_{t,x}^\alpha \nu\|_{L^2}^2,
\end{align*}
which concludes the proof.
\end{proof}

Now, we provide the estimates for the Navier-Stokes part $\eqref{A-2}_3$-$\eqref{A-2}_4$.

\begin{lemma}\label{L2.5}
For $T>0$, suppose that $\mh(s;T)\le \e ^2 \ll 1$ for a sufficiently small $\e >0$ satisfying
\[
\sup_{0\le t \le T} \|h(t)\|_{L^\infty} \le \frac12, \quad \sup_{0 \le t \le T} \|g(t)\|_{L^\infty} \le \log2.
\]
Then for each  $1 \le \ell\le s$ and $0<t \le T$,
\[
\frac{d}{dt}\lt(\| v\|_{\mathfrak{X}^\ell}^2 +\gamma \| h\|_{\mathfrak{X}^\ell}^2\rt) + \frac12\lt(\| v\|_{\mathfrak{X}^\ell}^2 + \| \nabla v\|_{\mathfrak{X}^\ell}^2\rt)\le C\e  \mathcal{S}(\ell;t) + C\e W(\ell;t) + CW(\ell-1;t),
\]
where $C=C(\ell,s,\gamma,\om)$ is independent of $T$.
\end{lemma}
\begin{proof}
Since the proof is lengthy and technical, we postpone it to Appendix A.
\end{proof}

Next, we investigate the dissipation estimates for $\nabla_{t,x}^\alpha h$.

\begin{lemma}\label{L2.6}
For $T>0$, suppose that $\mh(s;T)\le \e ^2 \ll 1$ for a sufficiently small $\e >0$ satisfying
\[
\gamma \min\lt\{ \lt(1+\mathfrak{h}\rt)^{\gamma-2}, \lt(1-\mathfrak{h}\rt)^{\gamma-2} \rt\} \ge \frac12, \quad \sup_{0 \le t \le T} \|g(t)\|_{L^\infty} \le \log2,
\]
where $\mathfrak{h}=\sup_{0\le t \le T} \|h(t)\|_{L^\infty}$. Then for each  $1 \le \ell\le s$ and $0<t \le T$,
\[
\begin{aligned}
\frac{d}{dt}&\lt( \sum_{|\alpha|\le \ell-1}\int_\om \nabla(\nabla^\alpha h)\cdot \nabla^\alpha v\,dx -\sum_{|\alpha|+m=\ell, \ m\ge 1} \int_\om \pa_t^{m-1}(\nabla^\alpha h) \pa_t^m (\nabla^\alpha h)\,dx\rt) + \frac14 \|h\|_{\mathfrak{X}^\ell}^2\\
&\le C\e   \mathcal{S}(\ell;t) + C\e  W(\ell;t) + C\lt(\|h\|_{\mathfrak{X}^{\ell-1}}^2 + \|u\|_{\mathfrak{X}^{\ell-1}}^2 + \|v\|_{\mathfrak{X}^\ell}^2+ \|\nabla v\|_{\mathfrak{X}^\ell}^2\rt),
\end{aligned}
\]
where $C=C(\ell,s,\gamma,\om)$ is independent of $T$.
\end{lemma}
\begin{proof}
First, for $\alpha$ and $m$ satisfying $|\alpha|+m=\ell$, $m\ge 1$ and $1 \le \ell \le s$, we get

\begin{align}
\begin{aligned}\label{L2-6.1}
-\frac{d}{dt}&\int_\om\pa_t^{m-1} (\nabla^\alpha h) \pa_t^m (\nabla^\alpha h)\,dx \\
&= -\|\pa_t^m (\nabla^\alpha h)\|_{L^2}^2 -\int_\om \pa_t^{m-1} (\nabla^\alpha h)\pa_t^{m+1}(\nabla^\alpha h)\,dx\\
&= -\|\pa_t^m (\nabla^\alpha h)\|_{L^2}^2 + \int_\om \pa_t^{m-1} (\nabla^\alpha h) \pa_t^m (\nabla^\alpha ( \nabla\cdot((1+h)v)))\,dx\\
&=  -\|\pa_t^m (\nabla^\alpha h)\|_{L^2}^2 -\int_\om \nabla(\pa_t^{m-1}(\nabla^\alpha h) \cdot \pa_t^m (\nabla^\alpha((1+h)v))\,dx\\
&\le  -\|\pa_t^m (\nabla^\alpha h)\|_{L^2}^2 +C\e   \mathcal{S}(\ell;t) + \int_\om \pa_t^{m-1}(\nabla^\alpha h)\nabla\cdot(\pa_t^m(\nabla^\alpha v))\,dx\\
&\le  -\|\pa_t^m (\nabla^\alpha h)\|_{L^2}^2 +C\e   \mathcal{S}(\ell;t) + \|\pa_t^{m-1}(\nabla^\alpha h)\|_{L^2}\|\nabla(\pa_t^m(\nabla^\alpha v))\|_{L^2},
\end{aligned}
\end{align}
where $C=C(\ell,s,\om)$ is independent of $T$. Then we sum relation \eqref{L2-6.1} over $m$ and $k$ to get
\bq\label{L2-6.2}
\begin{aligned}
-\frac{d}{dt}&\sum_{|\alpha|+m=\ell, \ m \ge 1}\int_\om \pa_t^{m-1}(\nabla^\alpha h)\pa_t^m(\nabla^\alpha h)\,dx + \sum_{|\alpha|+m=\ell, \ m \ge 1} \|\pa_t^m (\nabla^\alpha h)\|_{L^2}^2\\
&\le C\e   \mathcal{S}(\ell;t) + \|h\|_{\mathfrak{X}^{\ell-1}}^2 + \|\nabla v\|_{\mathfrak{X}^\ell}^2,
\end{aligned}
\eq
where $C=C(\ell,s,\om)$ is independent of $T$. Moreover, for $0 \le |\alpha| \le \ell-1$, we have

\begin{align*}
\frac{d}{dt}\int_\om \nabla(\nabla^{\alpha}  h) \cdot \nabla^\alpha v\,dx &= -\int_\om \nabla^\alpha (\pa_t h) \nabla \cdot (\nabla^\alpha v)\,dx\\
&\quad +\int_\om \nabla(\nabla^{\alpha}  h) \cdot \nabla^\alpha (\pa_t v)\,dx\\
&=: \mathcal{I}_1 + \mathcal{I}_2.
\end{align*}
For $\mathcal{I}_1$, one gets

\[\begin{aligned}
\mathcal{I}_1 &= \int_\om \nabla^\alpha (\nabla \cdot ((1+h)v)) \nabla \cdot (\nabla^\alpha v)\,dx\\
&=  \int_\om \nabla^\alpha (\nabla \cdot (hv))\nabla \cdot (\nabla^\alpha v)\,dx + \|\nabla \cdot (\nabla^\alpha v)\|_{L^2}^2\\
&\le C\e   \mathcal{S}(\ell;t) +  \|\nabla \cdot (\nabla^\alpha v)\|_{L^2}^2.
\end{aligned}\]
For $\mathcal{I}_2$, we use relation \eqref{ind_1} and \eqref{L2-5.1} to get

\begin{align*}
\mathcal{I}_2 &=-\int_\om \nabla(\nabla^{\alpha}  h) \cdot \nabla^\alpha \lt(v \cdot \nabla v + \gamma(1+h)^{\gamma-2}\nabla h -\frac{\Delta v}{1+h} + \frac{e^g}{1+h} (v-u)\rt) \,dx\\
&=: \sum_{i=1}^4 \mathcal{I}_2^i.
\end{align*}
For $\mathcal{I}_2^1$, straightforward computation gives

\[
\mathcal{I}_2^1 \le C\e  \mathcal{S}(\ell;t).
\]
For $\mathcal{I}_2^2$, one uses relation \eqref{ind_1} to obtain

\[\begin{aligned}
\mathcal{I}_2^2 &= -\gamma \sum_{\mu<\alpha}\binom{\alpha}{\mu} \int_\om \nabla(\nabla^{\alpha}  h) \cdot \nabla^\mu(\nabla h) \nabla^{\alpha-\mu}((1+h)^{\gamma-2})\,dx\\
&\quad - \gamma\int_\om |\nabla(\nabla^{\alpha} h)|^2 (1+h)^{\gamma-2}\,dx\\
&\le C\e  \mathcal{S}(\ell;t) -\frac12 \|\nabla(\nabla^{\alpha}  h)\|_{L^2}^2,
\end{aligned}\]
where we used
\[
\gamma(1+h)^{\gamma-2} \ge  \gamma\min\lt\{ \lt(1+\sup_{0\le t \le T} \|h(t)\|_{L^\infty}\rt)^{\gamma-2}, \lt(1-\sup_{0\le t \le T} \|h(t)\|_{L^\infty}\rt)^{\gamma-2} \rt\}\ge \frac{1}{2},
\]
and $C=C(\ell,s,\gamma,\om)$ is independent of $T$. \\

\noindent For $\mathcal{I}_2^3$, we again use relation \eqref{ind_1}  to yield

\[\begin{aligned}
\mathcal{I}_2^3 &= \sum_{\mu<\alpha} \int_\om \nabla(\nabla^{\alpha}  h) \cdot \nabla^{\mu}(\Delta v) \nabla^{\alpha-\mu}\lt(\frac{1}{1+h}\rt)\,dx\\
&\quad + \int_\om \nabla(\nabla^{\alpha}  h) \cdot \frac{\nabla^\alpha(\Delta v)}{1+h}\,dx\\
&\le C\e   \mathcal{S}(\ell;t) + C\|\nabla(\nabla^\alpha h)\|_{L^2} \|\nabla^\alpha (\Delta v)\|_{L^2}\\
&\le C\e   \mathcal{S}(\ell;t) +\frac18\|\nabla(\nabla^\alpha h)\|_{L^2}^2 + C \|\nabla^\alpha (\Delta v)\|_{L^2}^2.
\end{aligned}\]
For $\mathcal{I}_2^4$, we use relation \eqref{L2-5.1} and Young's inequality to have

\[\begin{aligned}
\mathcal{I}_2^4 &= -\sum_{\mu<\alpha}\binom{\alpha}{\mu}\int_\om \nabla(\nabla^{\alpha}  h) \cdot \nabla^\mu(v-u) \nabla^{\alpha-\mu}\lt(\frac{e^g}{1+h}\rt)\,dx\\
&\quad - \int_\om \frac{e^g}{1+h} \nabla^\alpha(v-u)\cdot \nabla(\nabla^{\alpha}  h)\,dx\\
&\le C\e  W(\ell;t) + C\e   \mathcal{S}(\ell;t) + C(\|\nabla^\alpha v\|_{L^2} + \|\nabla^\alpha u\|_{L^2})\|\nabla(\nabla^\alpha h)\|_{L^2}\\
&\le C\e  W(\ell;t) + C\e   \mathcal{S}(\ell;t) + \frac18\|\nabla(\nabla^\alpha h)\|_{L^2}^2 + C(\|\nabla^\alpha v\|_{L^2}^2 + \|\nabla^\alpha u\|_{L^2}^2).
\end{aligned}\]
Thus, we collect the estimates for $\mathcal{I}_2^i$'s to yield

\bq\label{L2-6.3}
\begin{aligned}
\frac{d}{dt}&\int_\om \nabla(\nabla^{\alpha}  h) \cdot \nabla^\alpha v\,dx + \frac14\|\nabla(\nabla^\alpha h)\|_{L^2}^2 \\
&\le C\e   \mathcal{S}(\ell;t) + C\e W(\ell;t)  + C\lt( \|\nabla^\alpha u\|_{L^2}^2 + \|\nabla^\alpha v\|_{L^2}^2 +\|\nabla\cdot(\nabla^\alpha v)\|_{L^2}^2 + \|\Delta (\nabla^\alpha v)\|_{L^2}^2\rt),
\end{aligned}
\eq
where $C=C(\ell,s,\om)$ is independent of $T$. We sum \eqref{L2-6.3} over every $\alpha$ with $0\le|\alpha|\le \ell-1$ to get

\bq\label{L2-6.4}
\begin{aligned}
\frac{d}{dt}&\sum_{ |\alpha|\le \ell-1} \int_\om \nabla(\nabla^\alpha h) \cdot \nabla^\alpha v\,dx + c_0 \|\nabla h\|_{H^{\ell-1}}^2\\
&\le C\e  \mathcal{S}(\ell;t) + C\e  W(\ell;t) + C\lt( \|u\|_{H^{\ell-1}}^2  + \|v\|_{H^\ell}^2 + \|\nabla v\|_{H^\ell}^2\rt),
\end{aligned}
\eq
where $C=C(\ell,s,\om)$ is independent of $T$. Moreover, we can use Poincar\'e inequality to get

\bq\label{L2-6.5}
\begin{aligned}
\|h\|_{L^2}^2 \le C\|\nabla h\|_{L^2}^2 &\le C\|\nabla h (1+h)^{\gamma-2}\|_{L^2}^2\\
&\le C\lt( \|\pa_t v\|_{L^2}^2 + \|v \cdot \nabla v\|_{L^2}^2 + \lt\|\frac{\Delta v}{1+h}\rt\|_{L^2}^2 + \lt\|\frac{e^g}{1+h}(v-u)\rt\|_{L^2}^2\rt)\\
&\le C\lt(\e  \mathcal{S}(1;t)+\|\pa_t v\|_{L^2}^2  +   \|\Delta v\|_{L^2}^2 + \|v\|_{L^2}^2 + \|u\|_{L^2}^2\rt)\\
&\le C\lt(\e  \mathcal{S}(1;t) + \|v\|_{\mathfrak{X}^1}^2 + \|\nabla v\|_{H^1}^2 + \|u\|_{L^2}^2\rt),
\end{aligned}
\eq
where $C$ is independent of $T$. Thus, we gather \eqref{L2-6.2}, \eqref{L2-6.4} and \eqref{L2-6.5} to yield the desired result.

\end{proof}

\vspace{0.2cm}

We gather all the estimates from previous lemmas to yield the following result.

\begin{corollary}\label{C2.1}
For $T>0$, suppose that $\mh(s;T)\le \e ^2 \ll 1$ for a sufficiently small $\e >0$ satisfying
\[
\e \le \min\{\e_p\}_{p=1}^s, \quad  \gamma \min\lt\{ \lt(1+\mathfrak{h}\rt)^{\gamma-2}, \lt(1-\mathfrak{h}\rt)^{\gamma-2} \rt\} \ge \frac12, \quad \mathfrak{h}\le \frac12, \quad \sup_{0 \le t \le T} \|g(t)\|_{L^\infty} \le \log2,
\]
where $\mathfrak{h}=\sup_{0\le t \le T} \|h(t)\|_{L^\infty}$. Then, there exists a constant $C^*=C^*(\ell,s,\gamma,\om)>0$ such that
\[
\mh(s;T)\le C^*\mh_0(s).
\]
\end{corollary}

\begin{proof}
Our claim is to prove the following assertion based on the induction:
\[
\mh(\ell;T) \le C\mh_0(\ell), \quad 0 \le \ell \le s,
\]
where $C$ is independent of $T$. Recall that the case $\ell=0$ is proved in Proposition \ref{P2.1}. Thus, it suffices to prove the induction part. So we assume that the following holds: there exists a constant $C>0$ independent of $T$ such that
\[
\mh(m;T) \le C\mh_0(m), \quad 0 \le m \le \ell-1.
\]
Then, we combine the estimates in Lemmas \ref{L2.5}-\ref{L2.6} to find
\bq\label{C2-1.1}
\begin{aligned}
\frac{d}{dt}&\tilde{\ms}(\ell;t) + \frac18 \lt( \ms(\ell;t) + \|\nabla v\|_{\mathfrak{X}^\ell}^2\rt)\\
&\le C\lt(\e  \ms(\ell;t) + \e  W(\ell;t) + \|h\|_{\mathfrak{X}^{\ell-1}}^2 + W(\ell-1;t)\rt),
\end{aligned}
\eq
where $C=C(\ell,s,\gamma,\om)$ is independent of $T$ and $\tilde{\ms}(\ell;t)$ is written as
\[
\tilde{\ms}(\ell;t) := \|v\|_{\mathfrak{X}^\ell}^2 + \gamma \|h\|_{\mathfrak{X}^\ell}^2 + \frac12\sum_{|\alpha|\le \ell-1}\into \nabla(\nabla^\alpha h)\cdot \nabla^\alpha v\,dx -\frac12 \hspace{-0.5cm}\sum_{|\alpha|+m=\ell, \ m\ge 1} \into \pa_t^{m-1} (\nabla^\alpha h) \pa_t^m (\nabla^\alpha h)\,dx,
\]
and we have the equivalence 
\bq\label{C2-1.1.2}
\frac14\ms(\ell;t)\le \tilde{\ms}(\ell;t) \le \frac74\ms(\ell;t),
\eq
since
\[
\begin{aligned}
\Bigg|&\sum_{|\alpha|\le \ell-1}\into \nabla(\nabla^\alpha h)\cdot \nabla^\alpha v\,dx - \sum_{|\alpha|+m=\ell, \ m\ge 1} \into \pa_t^{m-1} (\nabla^\alpha h) \pa_t^m (\nabla^\alpha h)\,dx\Bigg|\\
&\le \sum_{|\alpha|\le \ell-1}\|\nabla(\nabla^\alpha h)\|_{L^2}\|\nabla^\alpha v\|_{L^2} +  \sum_{|\alpha|+m=\ell, \ m\ge 1}\|\pa_t^{m-1} (\nabla^\alpha h)\|_{L^2} \|\pa_t^m (\nabla^\alpha h)\|_{L^2}\\
&\le \frac{1}{2} \lt(\|h\|_{H^\ell}^2 + \|v\|_{H^{\ell-1}}^2\rt) +  \sum_{|\alpha|+m=\ell, \ m\ge 1} \|\pa_t^m (\nabla^\alpha h)\|_{L^2}^2\\
&\le \frac32 \ms(\ell,t).
\end{aligned}
\]
Moreover, we gather the estimates in Lemmas \ref{L2.2}-\ref{L2.4} to get

\bq\label{C2-1.2}
\begin{aligned}
\frac{d}{dt}&\lt[\mathcal{T}(\ell;t) -\frac12\sum_{1 \le m \le \ell}\int_\om (\pa_t^{m-1}g) (\pa_t^m g)\,dx \rt] +\frac12\lt[ \sum_{1\le m \le \ell} \|\pa_t^\ell g\|_{L^2}^2 + \sum_{0\le \ell \le m}\|\pa_t^\ell u\|_{L^2}^2 \rt] \\
&\le C\e  W(\ell;t) + \frac38\sum_{m \le \ell-1}\lt( \|\pa_t^m u\|_{L^2}^2 + \|\pa_t^m v\|_{L^2}^2\rt)+ 4\|v\|_{\mathfrak{X}^\ell}^2,
\end{aligned}
\eq
where $C=C(\ell,s,\om)$ is independent of $T$. Here, we use Poincar\'e inequality and smallness of $\e $ to get

\bq\label{C2-1.3}
\begin{aligned}
\|g\|_{L^2}^2 = \|\log \rho\|_{L^2}^2 \le C\lt\|\frac{\nabla \rho}{\rho}\rt\|_{L^2}^2 &= C\|\nabla g\|_{L^2}^2\\
&\le C\lt( \|\pa_t u\|_{L^2}^2 + \|u \cdot \nabla u\|_{L^2}^2 + \|u\|_{L^2}^2 + \|v\|_{L^2}^2\rt)\\
&\le C\lt(  \|\pa_t u\|_{L^2}^2  + (1+\|\nabla u\|_{L^\infty}^2)\|u\|_{L^2}^2 + \|v\|_{L^2}^2\rt)\\
&\le C_1\lt( \|\pa_t u\|_{L^2}^2 + \|u\|_{L^2}^2 + \|v\|_{L^2}^2\rt),
\end{aligned}
\eq
where $C_1=C_1(s,\om)$ is independent of $T$. Without loss of generality, we assume that $C_1>1$. Then, we multiply \eqref{C2-1.3} by $\lambda_1 := \frac{1}{4C_1}$, and add it to \eqref{C2-1.2} so that 
\bq\label{C2-1.4}
\begin{aligned}
\frac{d}{dt}&\lt[\mathcal{T}(\ell;t) -\frac12\sum_{1 \le m \le \ell}\int_\om (\pa_t^{m-1}g) (\pa_t^m g)\,dx \rt] +\lambda_1\mathcal{T}(\ell;t) \\
&\le C_2\lt(\e  W(\ell;t) + \sum_{0 \le m \le \ell-1}\lt( \|\pa_t^m u\|_{L^2}^2 + \|\pa_t^m v\|_{L^2}^2\rt)+ \|v\|_{\mathfrak{X}^\ell}^2\rt),
\end{aligned}
\eq
where $C_2=C_2(\ell,s,\om)>0$ is independent of $T$.  Then, we multiply \eqref{C2-1.4} by $\lambda_2 := \frac{1}{16C_2}$ and add it to \eqref{C2-1.1} so that we can obtain 
\bq\label{C2-1.5}
\begin{aligned}
\frac{d}{dt}& \lt[\tilde{\ms}(\ell;t) + \lambda_2 \mathcal{T}(\ell;t)-\frac{\lambda_2}{2}\sum_{1\le m \le \ell} \int_\om (\pa_t^{m-1}g) (\pa_t^m g)\,dx  \rt] + \lambda_3 \lt(\ms(\ell;t) + \mathcal{T}(\ell;t) + \|\nabla v\|_{\mathfrak{X}^\ell}^2\rt)\\
&\le C\lt(\e  \ms(\ell;t) + \e  W(\ell;t) + \ms(\ell-1;t) + W(\ell-1;t)\rt),
\end{aligned}
\eq
where $\lambda_3 := \min\{\lambda_1\lambda_2, \frac{1}{16}\}$ and $C=C(\ell,s,\gamma,\om)>0$ is independent of $T$. Here, due to
\[
\lt|\sum_{1\le m \le \ell} \int_\om (\pa_t^{m-1}g) (\pa_t^m g)\,dx\rt| \le \frac12 \sum_{1 \le m \le \ell} \lt(\|\pa_t^{m-1}g\|_{L^2}^2 + \|\pa_t^{m}g\|_{L^2}^2\rt) \le \mathcal{T}(\ell;t),
\]
we have the following equivalence:

\bq\label{C2-1.6}
\frac12\mathcal{T}(\ell;t) \le \tilde{\mathcal T}(\ell;t) := \mathcal{T}(\ell;t) -\frac12\sum_{1\le m \le \ell} \int_\om (\pa_t^{m-1}g) (\pa_t^m g)\,dx \le \frac32\mathcal{T}(\ell;t).
\eq
Hence, we apply Lemma \ref{L2.1}, \eqref{C2-1.1.2}, \eqref{C2-1.6} and induction hypothesis to \eqref{C2-1.5} and get

\[
\begin{aligned}
\frac{d}{dt}&\lt[\tilde\ms(\ell;t) + \lambda_2 \tilde{\mathcal T}(\ell;t)\rt] + \frac{\lambda_3}{4}\lt(\tilde{\ms}(\ell;t) + \tilde{\mathcal T}(\ell;t) + \|\nabla v\|_{\mathfrak{X}^\ell}^2\rt)\\
&\le C_3\e  ({\tilde\ms}(\ell;t) + \tilde{\mathcal T}(\ell;t)  ) + C\mh(\ell-1;t)\\
&\le C_3\e  ({\tilde\ms}(\ell;t) + \lambda_4 \tilde{\mathcal T}(\ell;t)  ) + C\mh_0(\ell-1),
\end{aligned}
\]
where $C$ and $C_3$ depend on $\ell$, $s$, $\gamma$, and $\Omega$ but are independent of $T$. Since $\e $ is sufficiently small, we can assume that it satisfies $-\frac{\lambda_3}{4} + C_3\e  < -\frac{\lambda_3}{8}$. Therefore, we apply Gr\"onwall's lemma to obtain

\[
\tilde{\ms}(\ell;t) + \lambda_2 \tilde{\mathcal T}(\ell;t) \le\lt( \tilde{\ms}_0(\ell) + \lambda_2 \tilde{\mathcal T}_0(\ell)\rt)e^{-\frac{\lambda_3}{8}t} + C\mh_0(\ell-1)\lt(1-e^{-\frac{\lambda_3}{8}t}\rt),
\]
and we use Lemma \ref{L2.1}, \eqref{C2-1.1.2} and \eqref{C2-1.6} to get
\[
\mh(\ell;T) \approx \sup_{0\le t \le T} \lt(\tilde{\ms}(\ell;t) + \lambda_4\tilde{\mathcal T}(\ell;t) \rt) \le C\mh_0(\ell),
\]
where $C=C(\ell,s,\gamma,\om)$ is independent of $T$ and this completes the induction argument. 
\end{proof}

\subsubsection{Proof of Theorem \ref{T1.1}} Now, we are ready to prove Theorem \ref{T1.1}. We choose $\tilde{\e} := \min\{\e_0, \e \}$ where $\e_0$ is from Theorem \ref{thm_local}  when $T=1$ (for convenience) and $\e $ is given in Corollary \ref{C2.1}. Now, we assume that the initial data $(g_0, u_0, h_0, v_0)$ satisfies
\[
\mh_0(s) \le \frac{\tilde{\e}^2}{2(1+C^*)},
\]
where $C^*>0$ appears in Corollary \ref{C2.1}. Then, we define the lifespan of strong solutions to system \eqref{A-2} as follows:
\[
\widehat{T}:= \sup\{ t\ge 0 \ | \ \mh(s;t) < \tilde{\e}^2\}.
\]
First, Theorem \ref{thm_local} implies $\widehat{T}>0$. Assume for a contradiction that $\widehat{T}<\infty$. Then, the definition of $\widehat{T}$ and Corollary \ref{C2.1} implies
\[
\tilde{\e}^2 = \mh(s;\widehat{T}) \le C^*\mh_0(s) \le C^*\frac{\tilde{\e}^2}{2(1+C^*)} \le \frac{\tilde{\e}^2}{2} < \tilde{\e}^2,
\]
which gives a contradiction. Therefore, this implies $\widehat{T}=\infty$ and this concludes the proof.

\vspace{0.6cm}

\section{Large-time behavior of a global classical solution}\label{sec:4}
\setcounter{equation}{0}
In this section, we investigate the large-time behavior estimates of a global classical solution to \eqref{A-1}. First, we define the energy and dissipation functionals as
\[
\begin{aligned}
&\mathscr{E}(t) := \frac12\into \rho|u|^2\,dx + \frac12\into n|v|^2\,dx +\into \rho \int_{\rho_c}^\rho \frac{z-\rho_c}{z^2}\,dzdx + \into n\int_{n_c}^n \frac{z^\gamma-(n_c)^\gamma}{z^2}\,dzdx,\\
&\mathscr{D}(t) := \into |\nabla v|^2\,dx + \into \rho|u-v|^2\,dx.
\end{aligned}
\]
\begin{lemma}\label{L3.1}
Let $(\rho,u,n,v)$ be any global classical solution to system \eqref{A-1}. Then we have
\[
\frac{d}{dt}\mathscr{E}(t) + \mathscr{D}(t) =0, \quad t >0.
\]
\end{lemma}
\begin{proof}
Note that the following holds:
\[
\begin{aligned}
&\frac{d}{dt}\lt(\into \rho \int_{\rho_c}^\rho \frac{z-\rho_c}{z^2}\,dzdx\rt)= \frac{d}{dt}\lt(\into \rho\log\rho\,dx \rt),\\
&\frac{d}{dt}\lt(\into n\int_{n_c}^n \frac{z^\gamma-(n_c)^\gamma}{z^2}\,dzdx\rt) = \frac{d}{dt}\lt(\frac{1}{\gamma-1}\into n^\gamma\,dx\rt).
\end{aligned}
\]
Then, as we did in Proposition \ref{P2.1} direct computation yields the desired result.
\end{proof}
We can use the following lemma to get an equivalence between the energy functional $\mathscr{E}$ and the Lyapunov functional $\mathscr{L}$. For the proof, we refer to \cite{C0}. 

\begin{lemma}\cite{C0}\label{L3.2}
Let $r_0$, $\bar r$ and $\gamma \ge 1$ be given positive constants and define
\[
f(\gamma, r_0; r) := r\int_{r_0}^r \frac{z^\gamma -r_0^\gamma}{z^2}\,dz, \quad r \in [0,\bar r].
\]
Then, there exists a constant $C=C(\gamma, r_0, \bar r)>0$ independent of $r$ such that
\[
\frac{1}{C}(r-r_0)^2 \le f(\gamma,r_0;r)\le C(r-r_0)^2, \quad \forall r \in[0,\bar r].
\]
\end{lemma}
\begin{remark}
Since $\rho \in [0,\bar\rho]$ and $n \in [0,\bar n]$, as a direct corollary of Lemma \ref{L3.2}, we have
\[
\rho\int_{\rho_c}^\rho \frac{z-\rho_c}{z^2}\,dz \approx (\rho-\rho_c)^2, \quad n\int_{n_c}^n \frac{z^\gamma -(n_c)^\gamma}{z^2}\,dz \approx (n-n_c)^2,
\]
and hence, we get
\[
{\mathscr E}(t) \approx \mathscr{L}(t), \quad \forall t \ge 0.
\]
\end{remark}
Next, we show that the dissipation functional $\mathscr{D}(t)$ can bound the Lyapunov functional $\mathscr{L}$ without the evolution of densities $\rho$ and $n$.

\begin{lemma}\label{L3.3}
Let $(\rho,u,n,v)$ be any global classical solution to system \eqref{A-1} satisfying
\[
\rho \in [0,\bar\rho], \quad n \in [0,\bar n].
\]
Then we have
\[
\mathscr{L}^-(t) \le C\mathscr{D}(t), \quad \forall t \ge 0,
\]
where $C$ is a positive constant independent of $t$ and $\mathscr{L}^-(t)$ is defined as
\[
\mathscr{L}^-(t) := \mathscr{L}(t) - \into (\rho-\rho_c)^2\,dx -\into (n-n_c)^2\,dx.
\]
\end{lemma}
\begin{proof}
First, we use Young's inequality and Poincar\'e inequality to get

\[
\begin{aligned}
\into \rho|u-v|^2\,dx &= \into \rho|u|^2\,dx +\into \rho|v|^2\,dx -2\into \rho u\cdot v\,dx\\
&\ge \frac12\into \rho|u|^2\,dx -\into\rho|v|^2\,dx\\
&\ge \frac12 \into\rho|u|^2\,dx -\bar\rho \into |v|^2\,dx\\
&\ge \frac12 \into\rho|u|^2\,dx - C\into |\nabla v|^2\,dx,
\end{aligned}
\]
where $C = C(\bar\rho, \Omega)$ is a constant independent of $t$.  Thus, we can choose a small constant $c_0 \in (0,1)$ such that
\bq\label{L3-3.1}
c_0\into \rho|u-v|^2\,dx + \frac12\into |\nabla v|^2\,dx \ge \frac{c_0}{2} \into \rho|u|^2\,dx.
\eq
Moreover, we again use Poincar\'e inequality to get
\bq\label{L3-3.2}
\into n|v|^2\,dx \le \bar n \into |v|^2\,dx \le C\into |\nabla v|^2\,dx,
\eq 
where $C = C(\bar n, \Omega)$ is a constant independent of $t$. Thus, we  combine \eqref{L3-3.1} and \eqref{L3-3.2} to get

\[
\begin{aligned}
\mathscr{D}(t) &= (1-c_0)\into \rho|u-v|^2\,dx + \lt(c_0 \into \rho|u-v|^2\,dx + \frac12\into |\nabla v|^2\,dx \rt) + \frac12\into |\nabla v|^2\,dx\\
&\ge C \lt(\into \rho |u|^2\,dx + \into n|v|^2\,dx\rt),
\end{aligned}
\]
which completes the proof.

\end{proof}
To obtain dissipation estimates with respect to densities $\rho$ and $n$, we would use Bogovskii--type inequality in a bounded domain, as done in previous literature \cite{FZZ12, FP00, FNP01}. For detail, we refer to \cite{Bog80, BS90, Gal94}.

\begin{lemma}\label{L3.4}
Let $\Omega$ be a bounded, Lipschitz domain in $\R^d$, $d\ge 2$, $p,r\in(1,\infty)$ given numbers and $f \in \{L^p(\Omega)\ | \ \into f = 0\}$. Then, the equation
\bq\label{Poi}
\nabla \cdot \nu = f, \quad \nu|_{\pa\om} = 0,
\eq
admits a solution operator $\mathcal{B} : f \mapsto \nu$ satisfying the following properties:
\begin{enumerate}
\item
$\mathcal{B}$ is a linear operator from $L^p(\om)$ to $[W_0^{1,p}(\om)]^d$:
\[
\|\mathcal{B}[f]\|_{W^{1,p}(\om)} \le C\|f\|_{L^p(\om)},
\]
where $C = C(p,\om)$ is a positive constant.

\item
$\nu= \mathcal{B}[f]$ solves the equation \eqref{Poi}.
\item
If $f\in L^p(\om)$ can be written as $f = \nabla \cdot g$ for some $g \in L^r(\om)$ satisfying $g \cdot r \equiv 0$ on $\pa\om$, where $r=r(x)$ is the outward unit normal vector field to $\pa\om$, then we have
\[
\|\mathcal{B}[f]\|_{L^r(\om)} \le C\|g\|_{L^r(\om)},
\]
where $C=C(p,r,\om)$ is a positive constant.
\end{enumerate}
\end{lemma}
Then, we can set a perturbed energy functional ${\mathscr E}^{\sigma_1, \sigma_2}$ using the operator $\mathcal{B}$ in Lemma \ref{L3.4} with $p=r=2$ and $d=3$:

\[
{\mathscr E}^{\sigma_1, \sigma_2} := {\mathscr E}(t) -\sigma_1\into \rho u \cdot \mathcal{B}[\rho-\rho_c]\,dx -\sigma_2 \into nv \cdot \mathcal{B}[n-n_c]\,dx.
\]

\begin{lemma}\label{L3.5}
The perturbed energy functional ${\mathscr E}^{\sigma_1, \sigma_2}$ satisfies the following relation:
\[
\frac{d}{dt}{\mathscr E}^{\sigma_1, \sigma_2}(t) +\mathscr{D}^{\sigma_1, \sigma_2} (t) = 0, \quad t \ge 0,
\]
where $\mathscr{D}^{\sigma_1, \sigma_2}$ is written by
\[
\begin{aligned}
\mathscr{D}^{\sigma_1, \sigma_2} &:=\mathscr{D}(t) + \sigma_1\bigg(\into (\rho u\otimes u):\nabla \mathcal{B}[\rho-\rho_c]\,dx +\into (\rho-\rho_c)^2\,dx \\
&\hspace{2.5cm} -\into \rho(u-v)\cdot \mathcal{B}[\rho-\rho_c]\,dx +\into \rho u \cdot \mathcal{B}[\pa_t \rho]\,dx \bigg)\\
&\quad +\sigma_2 \bigg( \into (nv \otimes v) : \nabla \mathcal{B}[n-n_c]\,dx +\into (n^\gamma-n_c^\gamma)(n-n_c)\,dx \\
&\hspace{1.5cm} - \into \nabla v : \nabla \mathcal{B}[n-n_c]\,dx -\into \rho(v-u)\cdot \mathcal{B}[n-n_c]\,dx +\into nv\cdot \mathcal{B}[\pa_t n]\,dx\bigg).
\end{aligned}
\]
\end{lemma}
\begin{proof}
From the definition of ${\mathscr E}^{\sigma_1, \sigma_2}$ and Lemma \ref{L3.1}, it suffices to show that
\[
\mathscr{D}^{\sigma_1, \sigma_2}(t) = \mathscr{D} + \sigma_1 \frac{d}{dt} \into \rho u \cdot \mathcal{B}[\rho-\rho_c]\,dx + \sigma_2 \frac{d}{dt}\into nv \cdot \mathcal{B}[n-n_c]\,dx.
\]
First, we have
\[
\begin{aligned}
\frac{d}{dt}\into \rho u \cdot \mathcal{B}[\rho-\rho_c]\,dx &= \into \pa_t (\rho u)\cdot \mathcal{B}[\rho-\rho_c]\,dx + \into \rho u \cdot \mathcal{B}[\pa_t \rho]\,dx\\
&=: J_1^1 + J_1^2.
\end{aligned}
\]
For $J_1^1$, one has

\[
\begin{aligned}
&J_1^1= -\into \lt( \nabla \cdot (\rho u \otimes u) + \nabla \rho + \rho(u-v)\rt)\cdot \mathcal{B}[\rho-\rho_c]\,dx \\
&= \into (\rho u \otimes u):\nabla \mathcal{B}[\rho-\rho_c]\,dx +\into \rho(\rho-\rho_c)\,dx -\into \rho(u-v)\cdot\mathcal{B}[\rho-\rho_c]\,dx \\
&=  \into (\rho u \otimes u):\nabla \mathcal{B}[\rho-\rho_c]\,dx +\into (\rho-\rho_c)^2\,dx -\into \rho(u-v)\cdot\mathcal{B}[\rho-\rho_c]\,dx.
\end{aligned}
\]
Here, note that the homogeneous boundary condition of $\mathcal{B}[\rho-\rho_c]$ was used when integrating by parts. Similarly, we estimate
\[
\begin{aligned}
\frac{d}{dt}\into nv \cdot \mathcal{B}[n-n_c]\,dx &= \into \pa_t (nv)\cdot \mathcal{B}[n-n_c]\,dx + \into nv \cdot \mathcal{B}[\pa_t n]\,dx\\
&=: J_2^1 + J_2^2.
\end{aligned}
\]
For $J_2^1$,
\[
\begin{aligned}
J_2^1&=-\into \lt( \nabla \cdot (nv\otimes v) + \nabla(n^\gamma) -\Delta v + \rho(v-u)\rt)\cdot \mathcal{B}[n-n_c]\,dx\\
&= \into (nv\otimes v) : \nabla \mathcal{B}[n-n_c]\,dx + \into n^\gamma (n-n_c)\,dx -\into \nabla v : \nabla \mathcal{B}[n-n_c]\,dx \\
&\quad - \into \rho(v-u)\cdot \mathcal{B}[n-n_c]\,dx \\
&= \into (nv\otimes v) : \nabla \mathcal{B}[n-n_c]\,dx + \into (n^\gamma -n_c^\gamma)(n-n_c)\,dx -\into \nabla v : \nabla \mathcal{B}[n-n_c]\,dx \\
&\quad - \into \rho(v-u)\cdot \mathcal{B}[n-n_c]\,dx,
\end{aligned}
\]
which completes the proof.
\end{proof}

\subsection{Proof of Theorem \ref{T1.2}} Now, we are ready to prove Theorem \ref{T1.2}. Note that for sufficiently small $\sigma_1$ and $\sigma_2$, we have
\[
{\mathscr E}^{\sigma_1, \sigma_2}(t) \approx {\mathscr E}(t) \approx \mathscr{L}(t), \quad \forall t \ge 0.
\]
Hence, it suffices to show that for sufficiently small $\sigma_1$ and $\sigma_2$,
\[
C\mathscr{L}(t) \le \mathscr{D}^{\sigma_1, \sigma_2}(t), \quad \forall t \ge 0.
\]
for some constant $C>0$ independent of $t$, since the combination of above relation with the equivalence ${\mathscr E}^{\sigma_1, \sigma_2}(t)\approx \mathscr{L}(t)$ and Lemma \ref{L3.5}  implies that there exists a constant $C>0$ independent of $t$ such that
\[
\frac{d}{dt}{\mathscr E}^{\sigma_1,\sigma_2} (t) + C{\mathscr E}^{\sigma_1,\sigma_2} (t) \le 0.
\]
First, we rewrite $\mathscr{D}^{\sigma_1, \sigma_2}$ as

\begin{align*}
\mathscr{D}^{\sigma_1, \sigma_2} &:=\mathscr{D}(t) + \sigma_1\bigg(\into (\rho u\otimes u):\nabla \mathcal{B}[\rho-\rho_c]\,dx +\into (\rho-\rho_c)^2\,dx \\
&\hspace{2.5cm} -\into \rho(u-v)\cdot \mathcal{B}[\rho-\rho_c]\,dx -\into \rho u \cdot \mathcal{B}[\nabla\cdot( \rho u)]\,dx \bigg)\\
&\quad +\sigma_2 \bigg( \into (nv \otimes v) : \nabla \mathcal{B}[n-n_c]\,dx +\into (n^\gamma-n_c^\gamma)(n-n_c)\,dx \\
&\hspace{1.5cm} - \into \nabla v : \nabla \mathcal{B}[n-n_c]\,dx -\into \rho(v-u)\cdot \mathcal{B}[n-n_c]\,dx -\into nv\cdot \mathcal{B}[\nabla\cdot(nv)]\,dx\bigg)\\
&=: \sum_{i=1}^{10} K_i.
\end{align*}
We estimate $K_i$'s one by one as follows:\\

\noindent $\diamond$ (Estimates for $K_2$) We use Cauchy-Schwarz inequality, Young's inequality and Lemma \ref{L3.4} to obtain
\[
\begin{aligned}
K_2 &\ge -\sigma_1 \|u\|_{L^\infty} \|\rho u\|_{L^2}\|\nabla \mathcal{B}[\rho-\rho_c]\|_{L^2} \\
&\ge -C\sigma_1 (\bar\rho)^{1/2}\|u\|_{L^\infty}\lt(\into \rho |u|^2\,dx\rt)^{1/2} \lt(\into (\rho-\rho_c)^2\,dx\rt)^{1/2}\\
&\ge -C\sigma_1^{1/2} \into \rho |u|^2\,dx -\sigma_1^{3/2} \into (\rho-\rho_c)^2\,dx,
\end{aligned}
\]
where $C = C(\bar\rho, \|u\|_{L^\infty}, \Omega)$ is a positive constant.\\

\noindent $\diamond$ (Estimates for $K_4$) Cauchy-Schwarz inequality and Young's inequality yield
\[
\begin{aligned}
K_4 &\ge \sigma_1 \|\rho(u-v)\|_{L^2}\|\mathcal{B}[\rho-\rho_c]\|_{L^2}\\
&\ge -C\sigma_1 \lt(\into \rho|u-v|^2\,dx\rt)^{1/2}\lt(\into (\rho-\rho_c)^2\,dx\rt)^{1/2}\\
&\ge -C\sigma_1^{1/2} \into \rho|u-v|^2\,dx -\sigma^{3/2}\into (\rho-\rho_c)^2\,dx,
\end{aligned}
\]
where $C = C(\bar\rho,  \Omega)$ is a positive constant.\\

\noindent $\diamond$ (Estimates for $K_5$) Recall that $\rho u \cdot r \equiv 0$ on $\pa\om$. Thus, (3) of Lemma \ref{L3.4} implies
\[
\|\mathcal{B}[\nabla\cdot(\rho u)]\|_{L^2} \le C\|\rho u\|_{L^2},
\]
where $C=C(\Omega)$ is a positive constant. Thus,
\[
K_5 \ge -C\sigma_1 \|\rho u\|_{L^2}^2 \ge -C\sigma_1 \into \rho |u|^2\,dx,
\]
where $C=C(\bar\rho, \Omega)$ is a positive constant.\\

\noindent $\diamond$ (Estimates for $K_6$) Similarly to the estimates for $K_2$,
\[
K_6 \ge -C\sigma_2^{1/2} \into n|v|^2\,dx -\sigma_2^{3/2} \into (n-n_c)^2\,dx,
\]
where $C = C(\bar n, \|v\|_{L^\infty}, \Omega)$ is a positive constant.\\

\noindent $\diamond$ (Estimates for $K_7$) Note that
\[
f(x) := x^\gamma - n_c^\gamma -\gamma (n_c)^{\gamma-1}(n-n_c) \ge 0, \quad \forall x \in [0,\bar n].
\]
Thus,
\[
K_7 \ge C\sigma_2 \into (n-n_c)^2\,dx,
\]
where $C=C(\gamma, n_c)$ is a positive constant.\\

\noindent $\diamond$ (Estimates for $K_8$) We use Cauchy-Schwarz inequality and Young's inequality to have

\[
K_8 \ge -\sigma_2 \|\nabla v\|_{L^2} \|\nabla \mathcal{B}[n-n_c]\|_{L^2} \ge -C\sigma_2^{1/2}\into |\nabla v|^2\,dx - \sigma_2^{3/2} \into (n-n_c)^2\,dx,
\]
where $C= C(\om)$ is a positive constant.\\

\noindent $\diamond$ (Estimates for $K_9$ and $K_{10}$) Similarly to estimates for $K_4$ and $K_5$, we can get
\[
\begin{aligned}
&K_9 \ge -C\sigma_2^{1/2}\into \rho|u-v|^2\,dx -\sigma_2^{3/2}\into (n-n_c)^2\,dx,\\
&K_{10} \ge -C\sigma_2 \into n|v|^2\,dx,
\end{aligned}
\]
where $C= C(\bar n, \om)$ is a positive constant.\\

Now, we gather all the estimates for $K_i$'s to obtain

\bq\label{T1-2.1}
\begin{aligned}
\mathscr{D}^{\sigma_1, \sigma_2}(t) &\ge \mathscr{D}(t) - C\sigma_1^{1/2}(1+\sigma_1^{1/2}) \into \rho |u|^2\,dx -C\sigma_2^{1/2}(1+\sigma_2^{1/2}) \into n|v|^2\,dx  \\
&\quad - C\sigma_2^{1/2} \into |\nabla v|^2\,dx  -C(\sigma_1^{1/2} + \sigma_2^{1/2})\into \rho|u-v|^2\,dx\\
&\quad +(\sigma_1 - C\sigma_1^{3/2}) \into (\rho-\rho_c)^2\,dx + (\sigma_2 - C\sigma_2^{3/2}) \into (n-n_c)^2\,dx,
\end{aligned}
\eq
where $C=C(\bar\rho, \bar n, \|u\|_{L^\infty}, \|v\|_{L^\infty}, \gamma, n_c, \om)$ is a positive constant. Then, we apply Lemma \ref{L3.3} to \eqref{T1-2.1} and choose sufficiently small $\sigma_1$ and $\sigma_2$ to get

\[
\begin{aligned}
\mathscr{D}^{\sigma_1, \sigma_2}(t) &\ge \frac12\mathscr{D}(t) - C\sigma_1^{1/2}(1+\sigma_1^{1/2}) \into \rho |u|^2\,dx -C\sigma_2^{1/2}(1+\sigma_2^{1/2}) \into n|v|^2\,dx  \\
&\quad+ \frac12\mathscr{D}(t) - C\sigma_2^{1/2} \into |\nabla v|^2\,dx  -C(\sigma_1^{1/2} + \sigma_2^{1/2})\into \rho|u-v|^2\,dx\\
&\quad +(\sigma_1 - C\sigma_1^{3/2}) \into (\rho-\rho_c)^2\,dx + (\sigma_2 - C\sigma_2^{3/2}) \into (n-n_c)^2\,dx\\
&\ge  c\lt( \mathscr{L}^-(t) + \into (\rho-\rho_c)^2\,dx +\into (n-n_c)^2\,dx\rt) = c\mathscr{L}(t),
\end{aligned}
\]
where $c=c(\bar\rho, \bar n, \|u\|_{L^\infty}, \|v\|_{L^\infty}, \gamma, n_c, \om)$ is a positive constant and this completes the proof.\\

\vspace{0.4cm}

%As a corollary of Theorem \ref{T1.1} and \ref{T1.2}, we actually find out that the (unique) global classical solution we established in Theorem \ref{T1.1} converges to the equilibrium $(\rho_\infty, u_\infty, n_\infty, v_\infty) = (\rho_c, 0 , n_c, 0)$ in a high order Sobolev norm as $t \to \infty$.
%
%\begin{corollary}\label{C3.1}
%Suppose that $\delta>0$ is sufficiently small as in the proof of Theorem \ref{T1.1} and assume that the initial data $(g_0, u_0, h_0, v_0)$ satisfies
%\[
%\mh_0(s) \le \frac{\delta^2}{2(1+C^*)}.
%\]
%Then, we have
%\[
%\lt( \|g(t)\|_{\mathfrak{X}^s}^2 + \|u(t)\|_{\mathfrak{X}^s}^2 + \|h(t)\|_{\mathfrak{X}^s}^2 + \|v(t)\|_{\mathfrak{X}^s}^2\rt) \le Ce^{-\eta t}, \quad t \ge 0,
%\]
%where $C$ and $\eta$ are independent of $t$.
%\end{corollary}

\subsection{Proof of Corollary \ref{C2.1}} We would combine the result in Theorem \ref{T1.2} with the estimates in the proof of  Theorem \ref{T1.1} to get the desired result. First, note that
\[
\begin{aligned}
\rho\int_{\rho_c}^\rho \frac{z-\rho_c}{z^2}\,dz &= \rho\log\lt(\frac{\rho}{\rho_c}\rt) + \rho_c -\rho\\
&= \rho_c \lt(\frac{\rho}{\rho_c}\log\lt(\frac{\rho}{\rho_c}\rt) + 1-\frac{\rho}{\rho_c}\rt)\\
&= \rho_c \lt( g e^g + 1 - e^g\rt) \approx |g|^2.
\end{aligned}
\]
Thus, Theorem \ref{T1.2} implies

\bq\label{C3-1.1}
\lt(\|g(t)\|_{L^2}^2 + \|u(t)\|_{L^2}^2 + \|h(t)\|_{L^2}^2 + \|v(t)\|_{L^2}^2\rt)\le CE(t) \le Ce^{-\eta_0 t} \quad \forall t \ge 0,
\eq
where $C$ and $\eta_0$ are constants independent of $t$. Now, we can deduce from the proofs in Corollary \ref{C2.1} that for each $1 \le \ell \le s$ and $t\ge0$, 
\[
\begin{aligned}
\frac{d}{dt}& \lt[\tilde{\ms}(\ell;t) + \lambda_2 \tilde{\mathcal{T}}(\ell;t) \rt] +\frac{ \lambda_3}{4} \lt(\tilde{\ms}(\ell;t) + \tilde{\mathcal{T}}(\ell;t) + \|\nabla v\|_{\mathfrak{X}^\ell}^2\rt)\\
&\le C_3\e \lt(\tilde{\ms}(\ell;t) + \tilde{\mathcal{T}}(\ell;t)\rt) + C(\ms(\ell-1;t) + W(\ell-1;t)),
\end{aligned}
\]
where $C=C(\ell,s,\gamma,\om)$ is independent of $t$, $\lambda_2$, $\lambda_3$ and $C_3$ are given in Corollary \ref{C2.1} and  we used the equivalences
\[
\ms(\ell;t) \approx \tilde{\ms}(\ell;t), \quad \mathcal{T}(\ell;t) \approx \tilde{\mathcal{T}}(\ell;t), \quad \ms(\ell;t) + W(\ell;t) \approx \ms(\ell;t) + \mathcal{T}(\ell;t).
\]
Here, \eqref{C3-1.1} can be rewritten as
\[
\ms(0;t) + W(0;t) \le Ce^{-\eta_0 t}, \quad \forall t \ge 0.
\]
Since we assumed that $C_3\e < \frac{\lambda_3}{8}$, we use the induction argument on $\ell$ and the equivalences to conclude the proof.

%%%%%%%%%%%%%%%%%%%%%%%%%%%%%%%%%%%%%%%%%%%%%%%%%%%%%%%%%%%%%%%%%%%%%%%%%%%%%%%%%%%%%%%%%%%%%%%%%%%%%%%%%%%%%%%%%%%%%%%%%%%%%%%%%%%%%%%%%%%%%%
%
%
%    Acknowledgments
%
%
%%%%%%%%%%%%%%%%%%%%%%%%%%%%%%%%%%%%%%%%%%%%%%%%%%%%%%%%%%%%%%%%%%%%%%%%%%%%%%%%%%%%%%%%%%%%%%%%%%%%%%%%%%%%%%%%%%%%%%%%%%%%%%%%%%%%%%%%%%%%%%%%

\section*{Acknowledgments}
The author was supported by NRF grant (No. 2019R1A6A1A10073437). The author also appreciates Prof. Young-Pil Choi for fruitful discussions.

%\section*{Conflict of interest}
%The author declares that they have no conflict of interest.
%	

\appendix
\section{Proof of Lemma \ref{L2.5}}\label{app.A}
In this appendix, we present the proof of Lemma \ref{L2.5}. We estimate $\nabla_{t,x}^\alpha v$ and $\nabla_{t,x}^\alpha h$ separately as follows:\\

\noindent $\bullet$ (Step A: Estimates for $\nabla_{t,x}^\alpha v$): First, we have

\begin{align*}
\frac12\frac{d}{dt}\|\nabla_{t,x}^\alpha v\|_{L^2}^2 &= \int_\om \nabla_{t,x}^\alpha (v \cdot \nabla v)\cdot \nabla_{t,x}^\alpha v\,dx\\
&\quad - \int_\om \nabla_{t,x}^\alpha \lt(\frac{\nabla p(1+h)}{1+h}\rt) \cdot \nabla_{t,x}^\alpha v\,dx\\
&\quad + \int_\om \nabla_{t,x}^\alpha \lt(\frac{\Delta v}{1+h}\rt) \cdot \nabla_{t,x}^\alpha v\,dx\\
&\quad -\int_\om \nabla_{t,x}^\alpha \lt(\frac{e^g}{1+h} (v-u)\rt)\cdot \nabla_{t,x}^\alpha v\,dx\\
&=: \sum_{i=1}^4 \mathcal{L}_i.
\end{align*}
We  estimate $\mathcal{L}_i$'s one by one as follows:\\

\vspace{0.2cm}

\noindent $\diamond$ (Estimates for $\mathcal{L}_1$): Here, we get
\[
\begin{aligned}
\mathcal{L}_1 &=-\int_\om v \cdot \nabla (\nabla_{t,x}^\alpha v) \cdot \nabla_{t,x}^\alpha v\,dx -\sum_{\mu\le \alpha, \ |\mu|\neq 0} \binom{\alpha}{\mu} \int_\om \nabla_{t,x}^\mu v \cdot \nabla (\nabla_{t,x}^{\alpha-\mu} v) \cdot \nabla_{t,x}^\alpha v\,dx\\
&= \frac12 \int_\om (\nabla \cdot v) |\nabla_{t,x}^\alpha v|^2 \,dx -\sum_{\mu\le \alpha, \ |\mu|\neq 0} \binom{\alpha}{\mu} \int_\om \nabla_{t,x}^\mu v \cdot \nabla (\nabla_{t,x}^{\alpha-\mu} v) \cdot \nabla_{t,x}^\alpha v\,dx\\
&\le C\e   \mathcal{S}(\ell;t), 
\end{aligned}
\]
where $C$ is independent of $T$.\\

\noindent $\diamond$ (Estimates for $\mathcal{L}_2$): In this case, we have

\[
\begin{aligned}
\mathcal{L}_2 &= \int_\om \nabla_{t,x}^\alpha \lt( \frac{\gamma}{\gamma-1}(1+h)^{\gamma-1}\rt) \nabla \cdot (\nabla_{t,x}^\alpha v)\,dx\\
&= \gamma\int_\om \nabla_{t,x}^{\alpha^-} \lt((1+h)^{\gamma-2} \nabla_{t,x}^{\delta} h\rt)\nabla \cdot (\nabla_{t,x}^\alpha v)\,dx\\
&= \gamma \sum_{\mu\le \alpha^-, \ |\mu|\neq 0}\binom{\alpha^-}{\mu}\int_\om \nabla_{t,x}^\mu ((1+h)^{\gamma-2}) (\nabla_{t,x}^{\alpha-\mu} h) \nabla  \cdot (\nabla_{t,x}^\alpha v)\,dx\\
&\quad + \gamma\int_\om (1+h)^{\gamma-2} (\nabla_{t,x}^\alpha h) \nabla \cdot (\nabla_{t,x}^\alpha v)\,dx,
\end{aligned}
\]
where $\alpha^-$ and $\delta$ denote multi-indexes satisfying 
\bq\label{index_alp}
\alpha^- + \delta = \alpha, \quad |\delta|=1.
\eq
Here, we claim
\bq\label{ind_1}
\|\nabla_{t,x}^\alpha \lt( (1+h)^k \rt)\|_{L^2} \le C\sqrt{ \mathcal{S}(\ell;t)}, \quad 1 \le \ell \le s, \quad |\alpha|=\ell, \quad k\in\R\setminus\{ 0\},
\eq
where $C=C(\ell,k,\om)$ is independent of $T$. We argue by induction on $\ell$. If $|\alpha|=1$, then we use $\mh(s;T)\le \e^2  \ll1 $ to get
\[
\|\nabla_{t,x}^\alpha \lt( (1+h)^k \rt)\|_{L^2} = \|k(1+h)^{k-1} \nabla_{t,x}^\alpha h\|_{L^2} \le C\|\nabla_{t,x}^\alpha h\|_{L^2} \le C\sqrt{ \ms(1;t)}.
\]
Now, assume \eqref{ind_1} holds for every $\alpha$ with $|\alpha|=p$ and $1\le p \le \ell-1$. Then for $|\alpha|=\ell$, we use the induction hypothesis and Sobolev inequality to get
\[
\begin{aligned}
\|\nabla_{t,x}^\alpha ((1+h)^k)\|_{L^2} &=\lt\| k (1+h)^{k-1} \nabla_{t,x}^\alpha h + k\sum_{\mu\le\alpha^-, \ |\mu| \neq 0} \binom{\alpha^-}{\mu} \nabla_{t,x}^\mu ((1+h)^{k-1}) \nabla_{t,x}^{\alpha-\mu} h\rt\|\\
&\le  C\|\nabla_{t,x}^\alpha h\|_{L^2} +C\sum_{\mu<\alpha^-, \ |\mu|\neq0} \|\nabla_{t,x}^\mu ((1+h))^{k-1}\|_{L^4}\|\nabla_{t,x}^{\alpha-\mu} h\|_{L^4}\\
&\quad +C\|\nabla_{t,x}^\delta h\|_{L^\infty} \|\nabla_{t,x}^{\alpha^-} ((1+h)^{k-1})\|_{L^2}\\
&\le C\sqrt{ \mathcal{S}(\ell;t)} + C\sum_{\mu<\alpha^-, \ |\mu|\neq0} \|\nabla_{t,x}^\mu ((1+h))^{k-1}\|_{H^1}\|\nabla_{t,x}^{\alpha-\mu} h\|_{H^1}\\
&\le  C\sqrt{ \mathcal{S}(\ell;t)},
\end{aligned}
\]
which implies \eqref{ind_1}. Thus, we get
\[
\begin{aligned}
\mathcal{L}_2 &\le C \e \sqrt{\mathcal{S}(\ell;t)}\|\nabla (\nabla_{t,x}^\alpha v)\|_{L^2} +  \gamma\int_\om (1+h)^{\gamma-2} (\nabla_{t,x}^\alpha h) \nabla \cdot (\nabla_{t,x}^\alpha v)\,dx\\
&\le C\e   \mathcal{S}(\ell;t) + C\e  \|\nabla (\nabla_{t,x}^\alpha v)\|_{L^2}^2 + \gamma\int_\om (1+h)^{\gamma-2} (\nabla_{t,x}^\alpha h) \nabla \cdot (\nabla_{t,x}^\alpha v)\,dx,
\end{aligned}
\]
where $C=C(\ell,s,\gamma,\om)$ is independent of $T$.\\

\noindent $\diamond$ (Estimates for $\mathcal{L}_3$): For this,

\[
\begin{aligned}
\mathcal{L}_3 &= \int_\om \frac{\nabla_{t,x}^\alpha (\Delta v)}{1+h} \cdot \nabla_{t,x}^\alpha v\,dx + \sum_{\mu<\alpha} \binom{\alpha}{\mu} \int_\om \nabla_{t,x}^\mu (\Delta v) \nabla_{t,x}^{\alpha-\mu}\lt(\frac{1}{1+h}\rt) \cdot \nabla_{t,x}^\alpha v\,dx\\
&=: \mathcal{L}_3^1 + \mathcal{L}_3^2.
\end{aligned}
\]
For $\mathcal{L}_3^1$, one has

\[
\begin{aligned}
\mathcal{L}_3^1 &= -\int_\om \frac{|\nabla (\nabla_{t,x}^\alpha v)|^2}{1+h}\,dx + \int_\om \frac{\nabla h}{(1+h)^2} \cdot \nabla(\nabla_{t,x}^\alpha v)\cdot \nabla_{t,x}^\alpha v\,dx\\
&\le -\frac12\|\nabla(\nabla_{t,x}^\alpha v)\|_{L^2}^2 +C\|\nabla h\|_{L^\infty} \|\nabla(\nabla_{t,x}^\alpha v)\|_{L^2} \|\nabla_{t,x}^\alpha v\|_{L^2}\\
&\le -\frac12\|\nabla(\nabla_{t,x}^\alpha v)\|_{L^2}^2 +C\e  \sqrt{S(\ell;t)} \|\nabla(\nabla_{t,x}^\alpha v)\|_{L^2}.
\end{aligned}
\]
For $\mathcal{L}_3^2$, we use Sobolev inequality to get

\begin{align*}
\mathcal{L}_3^2 &= \int_\om \nabla_{t,x}^\alpha \lt(\frac{1}{1+h}\rt) \Delta v \cdot \nabla_{t,x}^\alpha v\,dx\\
&\quad + \sum_{\mu<\alpha, \ |\mu|\neq 0}\binom{\alpha}{\mu} \int_\om \nabla_{t,x}^\mu (\Delta v) \nabla_{t,x}^{\alpha-\mu}\lt(\frac{1}{1+h}\rt) \cdot \nabla_{t,x}^\alpha v\,dx\\
&\le \|\Delta v\|_{L^4} \lt\|\nabla_{t,x}^\alpha\lt(\frac{1}{1+h}\rt)\rt\|_{L^2}\|\nabla_{t,x}^\alpha v\|_{L^4}\\
&\quad - \sum_{\mu<\alpha, \ |\mu|\neq 0}\binom{\alpha}{\mu} \int_\om \nabla_{t,x}^\mu(\nabla v) : \nabla\lt[ \nabla_{t,x}^{\alpha-\mu}\lt( \frac{1}{1+h}\rt) \nabla_{t,x}^\alpha v\rt]\,dx\\
&\le C\lt(\e  + \sqrt{ \mathcal{S}(\ell;t)}\rt) \lt( \|\nabla(\nabla_{t,x}^\alpha v)\|_{L^2} + \|\nabla_{t,x}^\alpha v\|_{L^2}\rt),
\end{align*}
where $C=C(\ell,s,\om)$ is independent of $T$. Thus, we use the smallness of $\e $ to get
\[
\mathcal{L}_3 \le -\frac38 \|\nabla(\nabla_{t,x}^\alpha v)\|_{L^2}^2 + C\e  \mathcal{S}(\ell;t),
\]
where $C=C(\ell,s,\om)$ is independent of $T$. \\

\noindent $\diamond$ (Estimates for $\mathcal{L}_4$) Now, we have

\[\begin{aligned}
\mathcal{L}_4 &= -\int_\om \frac{e^g}{1+h}\nabla_{t,x}^\alpha(v-u)\cdot \nabla_{t,x}^\alpha v\,dx -\sum_{\mu\le\alpha, \ |\mu|\neq0} \binom{\alpha}{\mu} \nabla_{t,x}^\mu \lt( \frac{e^g}{1+h}\rt) \nabla_{t,x}^{\alpha-\mu} (v-u)\cdot \nabla_{t,x}^\alpha v\,dx\\
&=: \mathcal{L}_4^1 + \mathcal{L}_4^2.
\end{aligned}\]
For $\mathcal{L}_4^1$, we consider two cases; (i) $\nabla_{t,x}^\alpha = \nabla^\alpha$ (only $x$-derivatives) and (ii) $\nabla_{t,x}^\alpha = \partial_t^m \nabla^\beta$ with $m + |\beta| = |\alpha|$, $m\ge 1$.\\

\noindent $\circ$ (The case (i) for $\mathcal{L}_4^1$) In this case, we get
\[\begin{aligned}
\mathcal{L}_4^1 &=- \int_\om \frac{e^g}{1+h} |\nabla^\alpha v|^2\,dx + \int_\om \frac{e^g}{1+h}\nabla^\alpha  u \cdot \nabla^\alpha v \,dx \\
&= - \int_\om \frac{e^g}{1+h} |\nabla^\alpha v|^2\,dx - \int_\om \frac{e^g}{1+h}\nabla^{\alpha^-}  u \cdot \nabla^{\alpha^+} v \,dx -\into \nabla^\delta \lt(\frac{e^g}{1+h}\rt)\nabla^{\alpha^-}u \cdot \nabla^\alpha v\,dx\\
&\le -\frac13 \|\nabla^\alpha v\|_{L^2}^2  + C\|\nabla^{\alpha^-} u\|_{L^2}\|\nabla(\nabla^\alpha v)\|_{L^2} + \lt\| \nabla^\delta \lt(\frac{e^g}{1+h}\rt)\rt\|_{L^\infty}\|\nabla^{\alpha^-} u\|_{L^2}\|\nabla^\alpha v\|_{L^2}\\
&\le -\frac14 \|\nabla^\alpha v\|_{L^2}^2 + C\sqrt{W(\ell-1;t)}\|\nabla(\nabla^\alpha v)\|_{L^2} + CW(\ell-1;t),
\end{aligned}\]
where $\alpha^-$ is defined as \eqref{index_alp} and $\alpha^+$ is given by

\[
\alpha+ \delta = \alpha^+, \quad |\delta|=1.
\]

\noindent $\circ$ (The case (ii) for $\mathcal{L}_4^1$) Here, we use $m\ge1$ and integration by parts to obtain

\[\begin{aligned}
\mathcal{L}_4^1 &= -\int_\om \frac{e^g}{1+h}|\nabla_{t,x}^\alpha v|^2\,dx + \int_\om \frac{e^g}{1+h} \pa_t^m \nabla^\beta u \cdot \pa_t^m \nabla^\beta v\,dx\\
&= -\int_\om \frac{e^g}{1+h}|\nabla_{t,x}^\alpha v|^2\,dx - \int_\om \frac{e^g}{1+h} \lt[\pa_t^{m-1}\nabla^\beta\lt( u \cdot \nabla u + \nabla g + (u-v)\rt) \rt]\cdot \pa_t^m \nabla^\beta v\,dx\\
&\le -\frac13\|\nabla_{t,x}^\alpha v\|_{L^2}^2 + CW(\ell;t)\|\pa_t^m \nabla^\beta v\|_{L^2} + C(\|\pa_t^{m-1}\nabla^\beta u\|_{L^2} + \|\pa_t^{m-1} \nabla^\beta v\|_{L^2})\|\pa_t^m \nabla^\beta v\|_{L^2} \\
&\quad + \int_\om \nabla\lt(\frac{e^g}{1+h}\rt) \lt(\pa_t^{m-1}\nabla^\beta g\rt) \cdot \pa_t^m \nabla^\beta v\,dx + \int_\om \frac{e^g}{1+h}\lt(\pa_t^{m-1}\nabla^\beta g\rt) \nabla \cdot (\pa_t^m \nabla^\beta v)\,dx\\
&\le  -\frac13\|\nabla_{t,x}^\alpha v\|_{L^2}^2+ CW(\ell;t)\|\nabla_{t,x}^\alpha v\|_{L^2} + C(\|\pa_t^{m-1}\nabla^\beta u\|_{L^2} + \|\pa_t^{m-1} \nabla^\beta v\|_{L^2})\| \nabla_{t,x}^\alpha v\|_{L^2} \\
&\quad + C\e  \|\pa_t^{m-1}\nabla^\beta g\|_{L^2} \| \nabla_{t,x}^\alpha v\|_{L^2} + C\|\pa_t^{m-1}\nabla^\beta g\|_{L^2}  \|\nabla( \nabla_{t,x}^\alpha v)\|_{L^2}\\
&\le -\frac14 \|\nabla_{t,x}^\alpha v\|_{L^2}^2 + C\e  W(\ell;t) + CW(\ell-1;t) + C\sqrt{W(\ell-1;t)} \|\nabla( \nabla_{t,x}^\alpha v)\|_{L^2}.
\end{aligned}\]
In either case, we can get

\[
\mathcal{L}_4^1 \le -\frac14 \|\nabla_{t,x}^\alpha v\|_{L^2}^2 + C\e  W(\ell;t) + CW(\ell-1;t) + C\sqrt{W(\ell-1;t)} \|\nabla( \nabla_{t,x}^\alpha v)\|_{L^2}.
\]
For $\mathcal{L}_4^2$, we have

\[
\begin{aligned}
\mathcal{L}_4^2 &= -\sum_{\mu<\alpha, \ |\mu|\neq0} \binom{\alpha}{\mu} \int_\om \nabla_{t,x}^\mu \lt(\frac{e^g}{1+h}\rt) \nabla_{t,x}^{\alpha-\mu} (v-u)\cdot \nabla_{t,x}^\alpha v\,dx\\
&\quad -\int_\om \nabla_{t,x}^\alpha \lt(\frac{e^g}{1+h}\rt) (v-u)\cdot \nabla_{t,x}^\alpha v\,dx\\
&\le C\sum_{\mu<\alpha, \ |\mu|\neq0} \lt\|\nabla_{t,x}^\mu \lt(\frac{e^g}{1+h}\rt)\rt\|_{L^4} \|\nabla_{t,x}^{\alpha-\mu} (v-u)\|_{L^4} \|\nabla_{t,x}^\alpha v\|_{L^2}\\
&\quad +  (\|u\|_{L^\infty} + \|v\|_{L^\infty}) \lt\|\nabla_{t,x}^\alpha \lt(\frac{e^g}{1+h}\rt)\rt\|_{L^2}\|\nabla_{t,x}^\alpha v\|_{L^2}.
\end{aligned}
\]
Now, we claim that
\bq\label{ind_2}
\|\nabla_{t,x}^\alpha ( e^g)\|_{L^2} \le C\sqrt{W(\ell;t)} , \quad 1 \le \ell \le s, \quad |\alpha|=\ell,
\eq
where $C=C(\ell,s,\om)$ is independent of $T$. We also argue by induction on $\ell$. When $|\alpha| = 1$, we have

\[
\|\nabla_{t,x}^\alpha (e^g)\|_{L^2} = \| e^g\nabla_{t,x}^\alpha  g\|_{L^2} \le C\|\nabla_{t,x}^\alpha g\|_{L^2}  \le C \sqrt{W(\ell; t)}.
\]
Then, suppose \eqref{ind_2} holds for every $\alpha$ with $|\alpha|=p$ and $1 \le p \le \ell-1$. For $|\alpha|=\ell$, we can get

\begin{align*}
\|\nabla_{t,x}^\alpha (e^g)\|_{L^2} &=\lt\|\nabla_{t,x}^{\alpha^-} \lt(e^g \nabla_{t,x}^\delta g\rt) \rt\|_{L^2}\\
&= \lt\| e^g \nabla_{t,x}^\alpha g + \nabla_{t,x}^{\alpha^-}(e^g) \nabla_{t,x}^\delta g + \sum_{\mu < \alpha^-, \ |\mu|\neq0} \binom{\alpha^-}{\mu} \nabla_{t,x}^{\mu}(e^g) \nabla_{t,x}^{\alpha-\mu} g\rt\|_{L^2}\\
&\le C\|\nabla_{t,x}^\alpha g\|_{L^2}+ \| \nabla_{t,x}^\delta g\|_{L^\infty}\|\nabla_{t,x}^{\alpha^-}(e^g)\|_{L^2} + C\sum_{\mu < \alpha^-, \ |\mu|\neq0} \|\nabla_{t,x}^\mu(e^g)\|_{L^4}\|\nabla_{t,x}^{\alpha-\mu} g\|_{L^4}\\
&\le C\|\nabla_{t,x}^\alpha g\|_{L^2} + C\e \|\nabla_{t,x}^{\alpha^-}(e^g)\|_{L^2}+ C\sum_{\mu < \alpha^-, \ |\mu|\neq0} \|\nabla_{t,x}^\mu(e^g)\|_{H^1}\|\nabla_{t,x}^{\alpha-\mu} g\|_{H^1}\\
&\le C\sqrt{W(\ell;t)},
\end{align*}
which completes the proof of relation \eqref{ind_2}. Thus, we can combine \eqref{ind_1} with \eqref{ind_2} to yield

\bq\label{L2-5.1}
\begin{aligned}
\lt\|\nabla_{t,x}^\alpha\lt(\frac{e^g}{1+h}\rt)\rt\|_{L^2} &\le C\lt\|\nabla_{t,x}^\alpha \lt(\frac{1}{1+h}\rt)\rt\| + C\|\nabla_{t,x}^\alpha (e^g)\|_{L^2}\\
&\quad + C\sum_{\mu<\alpha, \ |\mu|\neq0} \|\nabla_{t,x}^\mu (e^g)\|_{L^4} \lt\|\nabla_{t,x}^{\alpha-\mu}\lt(\frac{1}{1+h}\rt)\rt\|_{L^4}\\
&\le C\lt(\sqrt{W(\ell;t)} + \sqrt{ \mathcal{S}(\ell;t)}\rt) + C\sum_{\mu<\alpha, \ |\mu|\neq0} \|\nabla_{t,x}^\mu (e^g)\|_{H^1} \lt\|\nabla_{t,x}^{\alpha-\mu}\lt(\frac{1}{1+h}\rt)\rt\|_{H^1}\\
&\le C\lt(\sqrt{W(\ell;t)} + \sqrt{ \mathcal{S}(\ell;t)}\rt),
\end{aligned}
\eq
where $C=C(\ell,s,\om)$ is independent of $T$. Thus, we apply the relation \eqref{L2-5.1} to $\mathcal{L}_4^2$ and get

\[
\mathcal{L}_4^2 \le C\e \lt(\sqrt{W(\ell;t)} + \sqrt{ \mathcal{S}(\ell;t)}\rt) \|\nabla_{t,x}^\alpha v\|_{L^2},
\]
where $C=C(\ell,s,\om)$ is independent of $T$. Hence, we obtain

\[
\mathcal{L}_4 \le -\frac14\|\nabla_{t,x}^\alpha v\|_{L^2}^2 +  C\e  W(\ell;t) + C\e   \mathcal{S}(\ell;t) + CW(\ell-1;t) + C\sqrt{W(\ell-1;t)} \|\nabla( \nabla_{t,x}^\alpha v)\|_{L^2}.
\]
We collect all the estimates for $\mathcal{L}_i$'s and use Young's inequality to get

\bq\label{est_v}
\begin{aligned}
\frac{d}{dt} &\|\nabla_{t,x}^\alpha v\|_{L^2}^2 + \frac12\|\nabla_{t,x}^\alpha v\|_{L^2}^2 + \frac58\|\nabla(\nabla_{t,x}^\alpha v)\|_{L^2}^2 \\
&\le C\e   \mathcal{S}(\ell;t) + C\e W(\ell;t) + CW(\ell-1;t) + \gamma\int_\om (1+h)^{\gamma-2} \nabla \cdot (\nabla_{t,x}^\alpha v) (\nabla_{t,x}^\alpha h)\,dx,
\end{aligned}
\eq
where $C$ is independent of $T$.\\

\noindent $\bullet$ (Step B: Estimates for $\nabla_{t,x}^\alpha h$): Straightforward computation gives

\begin{align*}
\frac12\frac{d}{dt} \|\nabla_{t,x}^\alpha h\|_{L^2}^2 &= -\int_\om  \nabla_{t,x}^\alpha \lt[ \nabla \cdot ((1+h)v)\rt] \nabla_{t,x}^\alpha h\,dx\\
&= \frac12 \int_\om (\nabla\cdot v)|\nabla_{t,x}^\alpha h|^2\,dx\\
&\quad - \sum_{\mu<\alpha}\binom{\alpha}{\mu}\int_\om  \lt[\nabla(\nabla_{t,x}^\mu h) \cdot \nabla_{t,x}^{\alpha-\mu} v \rt]\nabla_{t,x}^\alpha h\,dx\\
&\quad - \sum_{\mu\le\alpha, \ |\mu|\neq0} \binom{\alpha}{\mu} \int_\om  \lt[\nabla_{t,x}^\mu h \nabla\cdot(\nabla_{t,x}^{\alpha-\mu} v) \rt]\nabla_{t,x}^\alpha h\,dx\\
&\quad-\int_\om  (1+h)\nabla\cdot (\nabla_{t,x}^\alpha v) \nabla_{t,x}^\alpha h\,dx\\
&\le C\e   \mathcal{S}(\ell;t) -\int_\om (1+h)\nabla\cdot (\nabla_{t,x}^\alpha v)\nabla_{t,x}^\alpha h \,dx,
\end{align*}
where $C=C(\ell,s,\om)$ is independent of $T$. Now, we combine the above estimate with \eqref{est_v} to attain

\begin{align*}
\frac{d}{dt}&\lt( \|\nabla_{t,x}^\alpha v\|_{L^2}^2 + \gamma\|\nabla_{t,x}^\alpha h\|_{L^2}^2\rt) + \frac12\|\nabla_{t,x}^\alpha v\|_{L^2}^2 + \frac58\|\nabla(\nabla_{t,x}^\alpha v)\|_{L^2}^2\\
&\le C\e   \mathcal{S}(\ell;t) + C\e W(\ell;t) + CW(\ell-1;t) \\
&\quad + \gamma \int_\om \lt( (1+h)^{\gamma-2} - (1+h)\rt) \nabla \cdot (\nabla_{t,x}^\alpha v)  (\nabla_{t,x}^\alpha h)\,dx\\
&\le C\e   \mathcal{S}(\ell;t) + C\e W(\ell;t) + CW(\ell-1;t) + C\|\nabla_{t,x}^\alpha h\|_{L^2}\|h\|_{L^\infty}\|\nabla(\nabla_{t,x}^\alpha v)\|_{L^2}\\
&\le C\e   \mathcal{S}(\ell;t) + C\e W(\ell;t) + CW(\ell-1;t) + \frac18\|\nabla(\nabla_{t,x}^\alpha v)\|_{L^2}^2,
\end{align*}
and hence,
\bq\label{L2-5.2}
\begin{aligned}
\frac{d}{dt}&\lt(\|\nabla_{t,x}^\alpha v\|_{L^2}^2 +\gamma \|\nabla_{t,x}^\alpha h\|_{L^2}^2\rt) + \frac12\lt(\|\nabla_{t,x}^\alpha v\|_{L^2}^2 + \|\nabla(\nabla_{t,x}^\alpha v)\|_{L^2}^2\rt)\\
&\le C\e   \mathcal{S}(\ell;t) + C\e W(\ell;t) + CW(\ell-1;t),
\end{aligned}
\eq
where $C=C(\ell,s,\gamma,\om)$ is independent of $T$. Therefore, we sum the relation \eqref{L2-5.2} over every $\alpha$ with $0 \le |\alpha| \le \ell$ to get the desired result.

%%%%%%%%%%%%%%%%%%%%%%%%%%%%%%%%%%%%%%%%%%%%%%%%%%%%%%%%%%%
%
%
%
% \begin{thebibliography}{99}
%
%
%%%%%%%%%%%%%%%%%%%%%%%%%%%%%%%%%%%%%%%%%%%%%%%%%%%%%%%%%%%

\end{document}